\documentclass{amsart}
\usepackage{amsfonts,amssymb,amsmath,amsthm, latexsym}
\usepackage{url}
\usepackage{enumerate}
\usepackage{graphicx}
\usepackage{color}

\newtheorem{theorem}{Theorem}[section]
\newtheorem{lemma}[theorem]{Lemma}
\newtheorem{proposition}[theorem]{Proposition}
\newtheorem{corollary}[theorem]{Corollary}
\theoremstyle{definition}
\newtheorem{definition}[theorem]{Definition}
\newtheorem{remark}[theorem]{Remark}
\numberwithin{equation}{section}

\author{S. Bezuglyi, J. Kwiatkowski and R. Yassawi}
\address{Institute for Low Temperature Physics, Kharkov, Ukraine \\
The University of Computer Science and Economics, Olsztyn, Poland, and \\
Department of Mathematics, Trent University, Peterborough, Canada
}
\email{bezuglyi@ilt.kharkov.ua\\ jkwiat@mat.uni.torun.pl \\ryassawi@trentu.ca}
\thanks{ The third author is partially supported by an NSERC Discovery Grant. }

\keywords{Bratteli diagrams, Vershik maps}
\subjclass{Primary 37B10, Secondary 37A20}

\usepackage{amsthm,amssymb,amsmath,latexsym}

\newtheorem{example}[theorem]{Example}
\newcommand{\Z}{\mathbb Z}
\newcommand{\N}{\mathbb N}
\newcommand{\om}{\omega}
\newcommand{\wt}{\widetilde}
\newcommand{\ol}{\overline}

\begin{document}
\title{Perfect orderings on finite rank Bratteli diagrams}

%\author{\bf S.~Bezuglyi}
%\email{bezuglyi@ilt.kharkov.ua}
%\\
%{\bf J.~Kwiatkowski}
%\\
%The University of Computer Science and Economics, Olsztyn, Poland \\
%jkwiat@mat.uni.torun.pl
%\\
% {\bf R. Yassawi}\footnote{Partially supported by an NSERC grant.}\\
%Trent University, Peterborough, Canada\\
%ryassawi@trentu.ca}

\begin{abstract}
Given a Bratteli diagram $B$, we study the set $\mathcal O_B$ of all
possible orderings on $B$ and its subset
$\mathcal P_B$ consisting of \textit{perfect} orderings that produce
Bratteli-Vershik topological  dynamical systems (Vershik maps). We give necessary and sufficient conditions for  the ordering $\om$ to be perfect. On the other hand, a
wide class of non-simple Bratteli diagrams that do not admit Vershik
maps is explicitly described. In the case of finite rank Bratteli
 diagrams, we show that the existence of perfect orderings with a prescribed
number of extreme paths constrains significantly the values of the  entries of
the incidence matrices  and the structure of the diagram $B$. Our proofs are based on the new notions of  {\em skeletons} and {\em associated graphs}, defined and studied in the paper. For a Bratteli diagram $B$ of rank $k$, we endow the set $\mathcal O_B$ with product measure $\mu$ and prove that there is some $1 \leq j\leq k$ such that $\mu$-almost all orderings on $B$ have $j$ maximal and $j$ minimal paths. If $j$ is strictly greater than the number of minimal
 components that $B$ has, then $\mu$-almost all orderings are imperfect.

\end{abstract}

\maketitle

\section{Introduction}\label{Introduction}

Bratteli diagrams (Definition \ref{Definition_Bratteli_Diagram}) originally appeared in the theory of
$C^*$-algebras, and have turned out to be a very powerful and productive tool
for the study of dynamical systems in the measurable, Borel, and Cantor
setting. The importance of Bratteli diagrams in dynamics is based on the
remarkable results obtained in the pioneering works by Vershik,
Herman, Giordano, Putnam, and Skau \cite{vershik:1981},
\cite{herman_putnam_skau:1992},
\cite{giordano_putnam_skau:1995}.
During  the last two decades,  diverse aspects of Bratteli diagrams, and dynamical systems defined on their path spaces, have been
extensively studied, such as  measures invariant under the tail equivalence relation, measurable and continuous eigenvalues, entropy and orbit equivalence of these systems.  We refer to a
recent survey by Durand \cite{durand:2010} where the reader will find
more references on this subject.

A  Bratteli diagram $B$ can be thought of as a partial, recursive set of instructions for building a family of symbolic dynamical systems on $X_B$, the space of infinite paths on $B$. The $n$-th level of the diagram defines a clopen partition $\xi_n$ of  $X_B$, so that the diagram gives us a sequence of refining partitions of $X_B$. The information contained in $B$ also allows us to write  $\xi_n$ as a finite collection of {\em unordered} ``towers'', indexed by the vertices of the $n$-th level  of $B$. At this point, however,  we do not know the order of the elements in these towers.
The edge set at  the $(n+1)$-st level tells us how the partition $\xi_{n+1}$ is built from the  partition $\xi_n$, using a `cutting' method. In particular, if we see $k$ edges from the $n$-th level vertex $v'$ to the vertex $v$ of $(n+1)$-st level, this tells us that there are exactly $k$ copies of the $v'$-tower placed somewhere in the $v$-tower. The set  of edges with range $v$, denoted by $r^{-1}(v)$, thus contains all information about how many copies of  towers from $\xi_n$ we use to build the $v$-tower.

We can define  a homeomorphism on $X_B$ by putting a linear  order on the edges from $r^{-1}(v)$, which describes how we {\em stack} our level $n$ towers to get the level $(n+1)$ towers. We do this for each vertex $v$ and each level $n$. The  resulting partial  {\em order} $\om$ on $B$ (Definition \ref{order_definition}) admits a  map $\varphi_\om$ on $X_B$, where each point $x$  moves up the  tower to which it belongs. But what if  $x$ lives at the top of a tower for each level? In this case   $x$ is called a {\em maximal path}, and  it is on this set of maximal paths that we may not be able
 to extend the definition of $\varphi_\om$ so that it is continuous. We call an order $\om$  {\em perfect} if it admits a homeomorphism $\varphi_\om$ (called a {\em Vershik} or {\em adic} map) on $X_B$.  In this case each maximal path is sent to a {\em minimal} path: one that lives at the bottom of a tower for each level.
 The model theorem (Thm 4.7, \cite{herman_putnam_skau:1992}) tells us that every minimal\footnote{A minimal system $(X,T)$ is one which has no non-trivial proper subsystems: there is no closed, proper  $Y\subset X$ such that $T(Y) \subset Y$.} dynamical system on a Cantor space can be represented as a Bratteli-Vershik system $(X_B, \varphi_\om)$, where $B$ is a {\em simple} Bratteli diagram (Definition \ref{rank_d_definition}).
In \cite{medynets:2006} the model theorem is extended to
aperiodic homeomorphisms of a Cantor set  where the corresponding Bratteli diagrams are {\em aperiodic} (Definition \ref{aperiodic_diagram_definition}).

   Different orderings on $B$ generate different dynamical systems. In this article, we fix a Bratteli diagram $B$ and study the set $\mathcal O_B$ of all orderings on $B$, and its subset $\mathcal P_B$ of all perfect orderings  on $B$. We investigate the following questions:  Do there exist
simple criteria that would allow us to distinguish perfect and non-perfect
orderings?  Given a diagram $B$, and a natural number $j$, can one define a perfect order on $B$ with $j$ maximal paths? Which diagrams $B$ `support' no
perfect orders: i.e. when is  $\mathcal P_B $ empty? Given a Bratteli diagram $B$,   the set $\mathcal O_B$ can be represented
as a product space and the product topology turns it into a Cantor
set. It can also be endowed with a measure: since it is natural to assume that orders on
$r^{-1}(v)$ have equal probability, we  consider the uniformly
distributed product measure $\mu$ on $\mathcal O_B$. In this context,
the following questions are   interesting to us.
Given a Bratteli
diagram $B$, what can be said about the set $\mathcal O_B$ and its
subset $\mathcal P_B$ from the topological and measurable points of view?
It is worth commenting  here that we use in
this paper the term `ordering', instead of the more usual `order',  to
stress the difference between the case of ordered Bratteli diagrams,
when an order comes with the diagram, and Bratteli
diagrams with variable orderings, which is our context.

In Section \ref{Preliminaries}, we study general topological properties of $\mathcal O_B$. How `big' is $\mathcal P_B$ for a Bratteli diagram $B$? An order on $B$ is {\em proper} if it has a unique maximal path and a unique minimal path in $X_B$.
For a simple Bratteli diagram, the set of  proper orderings  is a nonempty subset of $\mathcal P_B$.\footnote{The family of proper orderings corresponds to generate strongly
orbit equivalent Vershik maps (Theorem 2.1, \cite{giordano_putnam_skau:1995} and Proposition 5.1, \cite{glasner_weiss:1995}).}
The relation $\mathcal
O_B = \mathcal P_B$ holds only for diagrams with one vertex at infinitely
many levels (Proposition \ref{Odometer}).
With this exception,   we show that in the case of most\footnote{We assume, without loss of generality,  that all {\em incidence matrix} entries are positive (see Definition \ref{incidence_matrices_definition}).} simple diagrams, the set of perfect orderings $\mathcal P_B$ and its complement are both dense
in $\mathcal O_B$ (Proposition \ref{GoodBadOrderingsAreDense}).
The case of non-simple Bratteli diagrams is more complicated. An example of a non-simple diagram $B$ such that
$\mathcal P_B =\emptyset$ was first found by Medynets in
\cite{medynets:2006}; in the present work, we clarify the essence of
Medynets' example, and describe a wide class of non-simple Bratteli
diagrams which support no perfect ordering in
Section \ref{no_perfection}.

 Can one decide whether a given order is perfect? We are interested mainly in the case when $\om$ is not proper.
 Suppose that  $B$  has the same vertex set $V$ at each level.
  When an ordering
$\om$ is chosen  on $B$,  then we can consider the
set of all words over the  alphabet $V$,   formed by sources  of  consecutive
finite paths \footnote{Consecutive finite paths are determined by the given order $\om$ on $B$} in $B$ which have  the same range. This set of words\footnote{Rather, the subset of  this set of words that are `seen' infinitely often.} defines the {\em language} of the ordered diagram $(B,\om)$ (Definition \ref{language}). We use the language  of $(B, \om)$ to
characterize whether or not $\om$ is perfect  (Proposition \ref{ExistenceVershikMap}), in terms of a permutation $\sigma$ of a finite set. This permutation encodes the action of $\varphi_\om$ on the set of maximal paths of $\om$, in this case a finite set.   For  {\em finite rank} Bratteli diagrams the number of vertices  at each level is bounded.  If $(B,\om)$ is an ordered finite rank diagram, it can be {\em telescoped}  (Definitions \ref{telescoping_definition} and  \ref{lexicographic_definition}) to an ordered diagram $(B', \om')$ where $B'$ has the same vertex set at each level. Since $(B,\om )$ is perfectly ordered if and only if $(B', \om')$ is perfectly ordered   (Lemma   \ref{good_orders_and_telescoping}), our described characterization of perfect orders in terms of a language can be used to verify whether any order on a finite rank diagram is perfect.
 As an example of how to apply these concepts, in Section \ref{odometer}, we  find
 sufficient conditions for a Bratteli-Vershik system $(X_B, \varphi_\om)$ to
be topologically conjugate to an  {\em odometer} (Definition \ref{odometer_definition}).

Next, we wish to study further the set $\mathcal P_B$. Let $\mathcal O_B(j)$ denote the set of orders with $j$ maximal paths. Given a finite rank  diagram $B$, when is $\mathcal O_B(j)\cap \mathcal P_B\neq \emptyset$? If $B$ has {\em rank $d$} (Definition \ref{rank_d_definition}), then $j$ must be at most $d$.  This problem is only interesting when $j>1$: if $B$ is simple, or if $B$ is aperiodic and generates dynamical systems with one minimal component\footnote{We use the term `minimal component' as a synonym to `minimal subset'. A dynamical system with $k$ minimal components has $k$ proper nontrivial minimal subsystems.}, then  $\mathcal O_B(1)\subset \mathcal P_B$, and it is simple to construct these orders. On the other hand,  if $B$ generates dynamical systems with $k$ minimal components, then  $\mathcal O_B(j)\cap \mathcal P_B= \emptyset$ for $j<k$.
 We mention a result from \cite{giordano_putnam_skau:1995}, first proved in \cite{putnam:1989}, where it is shown that if $\mathcal P_B \cap \mathcal O_B(j) \neq \emptyset$, then the {\em dimension group} of $B$ contains a copy of $\Z^{j-1}$ in its {\em infinitesimal subgroup} (see \cite{giordano_putnam_skau:1995} for definitions of these terms). However the proof of this result sheds little light on the structure of $B$. Given a finite rank diagram $B$, we attempt to construct  orders in $\mathcal P_B\cap \mathcal O_B(j)$ by constraining their  languages to behave as we would expect a perfect order's language to. Thus we fix a diagram $B$ with the same vertex set at each level, and
given  an integer $j$ between 2 and the rank of $B$,
we fix a permutation $\sigma$ of $\{1, \ldots ,j\}$. We then  create a framework to build perfect orderings $\om$ such that $\varphi_\om$ acts on the set of $\om$-maximal paths according to the instructions given by $\sigma$. We build such orderings by first specifying the set of all maximal edges in a certain way. This is the idea behind the notion of  a {\em skeleton}  $\mathcal F$ (Definition \ref{skeleton}), which partially defines an order. Given a skeleton and permutation, we define  a (directed) associated graph $\mathcal H$ (Definition \ref{associated_graph}). The graph $\mathcal H$, whose paths will correspond to words in the language of the putative perfect order, is used to take the partial instructions that we have been given by $\mathcal F$, and extend them to a perfect order on $B$.  Whether a perfect order exists on $B$ with a specified skeleton, depends on whether the {\em incidence matrices} of $B$ (Definition \ref{incidence_matrices_definition})
are related according to Theorem  \ref{existence of good order}. The
 simplest  case is if $B$ a simple, rank $d$ diagram and
  $\mathcal O_B(d) \cap \mathcal P_B\neq \emptyset$. Then   $B$'s incidence matrices $(F_n)$ are almost completely determined, as is the dynamical behaviour of the corresponding $\varphi_\om$ (Theorem \ref{d
 max/min paths}). A consequence of  Theorem  \ref{existence of good order}
 and  Remark
 \ref{non-simple case}, along with the fact that
 aperiodic Cantor homeomorphisms can be represented as adic systems, is that non-minimal
 aperiodic dynamical systems do not exist in abundance.
We remark that these notions can be generalized to non-finite  rank diagrams; however  the corresponding definitions are more technical, especially notationally.

In Section \ref{genericity_results}, we endow the set $\mathcal O_B$ with the uniform product
measure, and study questions about the measure of specific subsets of
$\mathcal O_B$. The results of this section are independent of those in Sections
\ref{language_skeleton} and \ref{characterization}.
We show, in Theorem \ref{goodbaddichotomy} and Corollary
\ref{j=j'},
that for a finite rank $d$ diagram there is
some $1 \leq j \leq d$ such that almost all orderings have exactly $j$
maximal and $j$ minimal paths.
 Whether for diagrams with
isomorphic dimension groups the $j$ is the same is an open question. In particular, in this section we cannot freely telescope our diagram: if $B'$ is a telescoping of $B$, then $\mathcal O_B$ is a set of 0 measure in $\mathcal O_B'$. We give necessary and sufficient conditions, in
terms of the incidence matrices of $B$, for verifying the value of
$j$, and show that $j=1$ for a large class of diagrams which include
linearly recurrent diagrams. We show in Theorem \ref{generic_two} that if $B$ is simple and $j>1$,
then a random ordering is not perfect.

We end with some questions. If $B'$ is a telescoping of $B$, how do $\mathcal P_B$ and $\mathcal P_B'$ compare?
 Do  Bratteli diagrams that support non-proper, perfect orders have special spectral properties? Do their dimension groups have any additional structure? Can one identify any interesting topological factors? Do these results generalize in some way to non-finite rank diagrams? If $B$ has finite rank and almost all orders on $B$ have $j$ maximal paths, is $j$ invariant under telescoping?

  \section{Bratteli diagrams and Vershik maps}\label{Preliminaries}

\subsection{Main definitions on Bratteli diagrams}

In this section, we collect the notation and basic definitions that
are used throughout the paper.
More information about Bratteli diagrams can be found in the papers
\cite{herman_putnam_skau:1992}, \cite{giordano_putnam_skau:1995},
\cite{durand_host_scau:1999}, \cite{medynets:2006},
\cite{bezuglyi_kwiatkowski_medynets:2009},
\cite{bezuglyi_kwiatkowski_medynets_solomyak:2010},
\cite{durand:2010} and references therein.

\begin{definition}\label{Definition_Bratteli_Diagram}
A {\it Bratteli diagram} is an infinite graph $B=(V^*,E)$ such that the vertex
set $V^*=\bigcup_{i\geq 0}V_i$ and the edge set $E=\bigcup_{i\geq 1}E_i$
are partitioned into disjoint subsets $V_i$ and $E_i$ where

(i) $V_0=\{v_0\}$ is a single point;

(ii) $V_i$ and $E_i$ are finite sets;

(iii) there exists a range map $r$ and a source map $s$, both from $E$ to
$V^*$, such that $r(E_i)= V_i$, $s(E_i)= V_{i-1}$, and
$s^{-1}(v)\neq\emptyset$, $r^{-1}(v')\neq\emptyset$ for all $v\in V^*$
and $v'\in V^*\setminus V_0$.
\end{definition}

The pair $(V_i,E_i)$ or just $V_i$ is called the $i$-th level of the
diagram $B$. A finite or infinite sequence of edges $(e_i : e_i\in E_i)$
such that $r(e_{i})=s(e_{i+1})$ is called a {\it finite} or {\it infinite
path}, respectively. For $m<n$, $v\, \in V_{m}$ and $w\,\in V_{n}$, let $E(v,w)$ denote the set of all paths $\overline{e} = (e_{1},\ldots, e_{p})$
with $s(e_{1})=v$ and $r(e_{p})=w$. If $m>n$ let $E(n,m)$ denote all paths whose source belongs to $V_n$ and whose range belongs to $V_m$.
For a Bratteli diagram $B$,
let $X_B$ be the set of infinite paths starting at the top vertex $v_0$.
 We
endow $X_B$ with the topology generated by cylinder sets
$\{U(e_j,\ldots,e_n): j, \,\, n \in \mathbb N, \mbox{ and }
 (e_j,\ldots,e_n) \in E(v, w),  v \in V_{j-1}, w \in V_n
\}$, where
$U(e_j,\ldots,e_n):=\{x\in X_B : x_i=e_i,\;i=j,\ldots,n, \,\,
%v=s(x_{j})\in V_{j-1} \mbox{  and }w = r(e_n) \in V_n
\}$.
 With this topology, $X_B$ is a 0-dimensional compact metric space.
We will consider such diagrams $B$ for which the path space $X_B$ has
no isolated points.  Letting $|A|$ denote the cardinality of the set $A$, this means that for every $(x_1,x_2,\ldots)\in X_B$
and every $n\geq 1$ there exists $m>n$ such that $|s^{-1}(r(x_m))|>1$.

\begin{definition}\label{incidence_matrices_definition}
Given a Bratteli diagram $B$, the $n$-th {\em incidence matrix}
$F_{n}=(f^{(n)}_{v,w}),\ n\geq 0,$ is a $|V_{n+1}|\times |V_n|$
matrix whose entries $f^{(n)}_{v,w}$ are equal to the number of
edges between the vertices $v\in V_{n+1}$ and $w\in V_{n}$, i.e.
$$
 f^{(n)}_{v,w} = |\{e\in E_{n+1} : r(e) = v, s(e) = w\}|.
$$
\end{definition}

 Observe that every vertex $v\in V^* $ is
connected to $v_0$ by a finite path and the set $E(v_0,v)$ of all such
paths is finite. Set $h_v^{(n)}=|E(v_0,v)|$ for $v\in V_{n}$. Then
\begin{equation*}
h_v^{(n+1)}=\sum_{w\in V_{n}}f_{v,w}^{(n)}h^{(n)}_w \ \mbox{or} \ \
h^{(n+1)}=F_{n}h^{(n)}
\end{equation*}
where $h^{(n)}=(h_w^{(n)})_{w\in V_n}$.

Next we define some popular families of Bratteli diagrams that we work with in this article.

\begin{definition}\label{rank_d_definition} Let $B$ be a Bratteli diagram.
\begin{enumerate}
\item
We say $B$ has \textit{finite rank} if  for some $k$, $|V_n| \leq k$ for all $n\geq 1$.
\item
Let $B$ have finite rank. We say $B$ has \textit{rank $d$} if  $d$ is the smallest
integer such that $|V_n|=d$ infinitely often.
\item
We say that $B$ is {\em simple} if  for any level
$n$ there is $m>n$ such that $E(v,w) \neq \emptyset$ for all $v\in
V_n$ and $w\in V_m$.
\item
 We say $B$ is  \textit{stationary} if $F_n = F_1$  for all $n\geq 2$.
\end{enumerate}
\end{definition}

\begin{definition}
For a Bratteli diagram $B$, the {\it tail (cofinal) equivalence}  relation
$\mathcal E$ on the path space $X_B$ is defined as $x\mathcal Ey$ if
 $x_n=y_n$ for all $n$ sufficiently large, where $x = (x_n)$, $y= (y_n)$.
\end{definition}

Let $X_{per}=\{x\in X_B : |[x]_{\mathcal E}| <\infty\}$. By definition,
we have $X_{per}=\{x\in X_B : \exists n>0 \mbox{ such that }(|r^{-1}(r(x_i))|=
1\ \forall i\geq n)\}$.

\begin{definition}\label{aperiodic_diagram_definition}
A Bratteli diagram $B$ is called
\textit{aperiodic} if $X_{per}=\emptyset$, i.e., every
$\mathcal  E$-orbit is countably infinite.
\end{definition}

We shall  constantly use the \textit{telescoping}
procedure for a Bratteli diagram:

\begin{definition} \label{telescoping_definition}
Let $B$ be a Bratteli diagram, and $n_0 = 0 <n_1<n_2 < \ldots$ be a strictly increasing sequence of integers. The {\em telescoping of $B$ to $(n_k)$} is the Bratteli diagram $B'$, whose $k$-level vertex set $V_k'= V_{n_k}$ and whose incidence matrices $(F_k')$ are defined by
\[F_k'= F_{n_{k+1}-1} \circ \ldots \circ F_{n_k},\]
where $(F_n)$ are the incidence matrices for $B$.
\end{definition}

Roughly speaking, in order to
telescope a Bratteli diagram, one takes a subsequence of levels
$\{n_k\}$ and considers the set $E(n_k, n_{k+1})$
 of all finite paths between the
 levels $\{n_k\}$ and $\{n_{k+1}\}$ as edges of the new diagram.
In particular, a Bratteli diagram $B$ has rank $d$ if and only if there is
a telescoping $B'$ of $B$ such that $B'$ has exactly $d$ vertices
at each level. When telescoping diagrams, we often do not specify to which levels $(n_k)$ we telescope, because it suffices to know that such a sequence of levels exists.

\begin{lemma}\label{aperiodic-property}
Every aperiodic Bratteli diagram $B$ can be telescoped to a diagram $B'$
with the property: $|r^{-1}(v)| \geq 2,\ v \in V^*\setminus V_0 $ and
$|s^{-1}(v)| \geq 2, \ v\in V^* \setminus V_0$.
\end{lemma}

In other words, we can state that, for any aperiodic Bratteli diagram,
 the properties $|r^{-1}(v)| \geq 2,\ v \in V^*\setminus V_0$, and
$|s^{-1}(v)| \geq 2, \ v\in V^* \setminus  V_0$, hold for infinitely
many levels $n$.

\begin{proof} We shall show that any periodic diagram $B$ can be telescoped so that
$|r^{-1}(v)| \geq 2,\ v \in V^*\setminus V_0$; the proof of the other statement is similar.
We need to show that for every $n\in \N$ there exists
$m > n$ such that for each vertex $v\in V_m$ there are at least two
finite paths $e,f \in E(n,m)$ with
$r(e) =r(f) =v$.
Assume that the converse is true. Then there exists $n$ such that for
all $m>n$ the set $U_m = \{x = (x_i)\in X_B : |r^{-1}(r(x_i))|=1,\
i= n+1,...,m\}$ is not empty. Clearly, $U_m$ is a clopen subset of $X_B$
and $U_m\supset U_{m+1}$. It follows that $X_{per} \supset U =
\bigcap_{m>n} U_m \neq \emptyset$.  This contradicts
the aperiodicity of the diagram.
\end{proof}

We will assume that the following \textit{convention} always holds:
{\em our diagrams are not disjoint unions of two subdiagrams.}
Here $B= (V^*,E)$ is a {\it disjoint union  of $B^1 =(V^{*,1}, E^1)$ and $B^2 =
(V^{*,2}, E^2)$} if  $V^*=V^{*,1}\cup V^{*,2}$,
  $V^{*,1}\cap V^{*,2} = \{v_{0}\}$ and  $E=E^{1}\sqcup E^{2}$.

\textit{Throughout the paper, we only consider aperiodic Bratteli diagrams $B$.  For
these diagrams $X_B$ is a Cantor set and $\mathcal E$ is a Borel
equivalence relation on $X_B$ with countably infinitely many equivalence classes.}

\begin{remark}
Given an aperiodic dynamical system $(X,T)$, a Bratteli diagram
is constructed by a sequence of Kakutani-Rokhlin partitions generated
by $(X,T)$ (see \cite{herman_putnam_skau:1992} and \cite{medynets:2006}).
The $n$-th level of the diagram corresponds to the $n$-th
Kakutani-Rokhlin partition  and the number $h_w^{(n)}$ is the height of
the $T$-tower labeled by the symbol $w$ from that partition.
\end{remark}

%
%%%%%%%%%%%%%Orderings and Vershik maps
%

\subsection{Orderings on a Bratteli diagram}

Let $B$ be a Bratteli diagram whose path space $X_B$
is a Cantor set.

\begin{definition}\label{order_definition} A Bratteli diagram $B=(V^*,E) $ is called {\it ordered}
if a linear order `$>$' is defined on every set  $r^{-1}(v)$, $v\in
\bigcup_{n\ge 1} V_n$. We use  $\om$ to denote the corresponding partial
order on $E$ and write $(B,\om)$ when we consider $B$ with the ordering $\om$. Denote by $\mathcal O_{B}$ the set of all orderings on $B$.
\end{definition}

Every $\omega \in \mathcal O_{B}$ defines the  \textit{lexicographic}
ordering on the set $E(k,l)$
 of finite paths between
vertices of levels $V_k$ and $V_l$:  $(e_{k+1},...,e_l) > (f_{k+1},...,f_l)$
if and only if there is $i$ with $k+1\le i\le l$, $e_j=f_j$ for $i<j\le l$
and $e_i> f_i$.
It follows that, given $\om \in \mathcal O_{B}$, any two paths from $E(v_0, v)$
are comparable with respect to the lexicographic ordering generated by $\om$.  If two infinite paths are tail equivalent, and agree from the vertex $v$ onwards, then we can compare them by comparing their initial segments in $E(v_0,v)$. Thus $\om$ defines a partial order on $X_B$, where two infinite paths are comparable if and only if they are tail equivalent.

   \begin{definition}
We call a finite or infinite path $e=(e_i)$ \textit{ maximal (minimal)} if every
$e_i$ is maximal (minimal) amongst the edges
from $r^{-1}(r(e_i))$.
\end{definition}

Notice that, for $v\in V_i,\ i\ge 1$, the
minimal and maximal (finite) paths in $E(v_0,v)$ are unique. Denote
by $X_{\max}(\om)$ and $X_{\min}(\om)$ the sets of all maximal and
minimal infinite paths in $X_B$, respectively. It is not hard to show that
$X_{\max}(\om)$ and $X_{\min}(\om)$ are \textit{non-empty closed subsets} of
$X_B$; in general, $X_{\max}(\om)$ and $X_{\min}(\om)$ may have
interior points. For a finite rank Bratteli diagram $B$, the sets $X_{\max}(\om)$
and $X_{\min}(\om)$ are always finite for any $\om$, and if $B$ has rank $d$,
then each of them have at most $d$ elements
(Proposition 6.2 in \cite{bezuglyi_kwiatkowski_medynets:2009}).

\begin{definition}
An  ordered Bratteli diagram $(B, \om)$ is called \textit{properly ordered}
if the sets $X_{\max}(\om)$ and $X_{\min}(\om)$ are singletons.
\end{definition}
We denote by $\mathcal O_B(j)$ the set of all orders on $B$ which have $j$ maximal paths.
Thus $\mathcal O_B(1)$ is the set of proper orders, and  if $B$ has rank $d$, then $\mathcal O_B = \bigcup_{j=1}^{d}\mathcal O_B(j)$.

\begin{definition}\label{lexicographic_definition}Let $(B,\omega)$ be an ordered Bratteli diagram, and suppose that $B'=(V',E')$ is the telescoping of $B$ to levels $(n_k)$. Let $v' \in V'$ and suppose that the two edges $e_1',$ $e_2'$, both with range $v'$, correspond to the finite paths $e_1$, $e_2$ in $B$, both with range $v$. Define the order $\om'$ on $B'$ by $e_1'<e_2'$ if and only if $e_1<e_2$. Then $\om' $ is called the {\em lexicographic order generated by $\om$} and is denoted by $\om'=L(\om)$.
\end{definition}

 It is not hard to see that if $\om' = L(\om)$, then
$$
|X_{\max}(\om)| = |X_{\max}(\om')|,\ \ |X_{\min}(\om)| = |X_{\min}(\om')|.
$$

Let $(B,\om)$ be an ordered Bratteli diagram. Then $x\in
X_{\max}(\om)\cap X_{\min}(\om)$ if and only if $|\mathcal E(x)|=1$.
Thus, if $B$ is an aperiodic Bratteli diagram, then $X_{\max}(\om)\cap
X_{\min}(\om) =\emptyset$.

\begin{definition}
Let $B$ be a stationary diagram. We say an ordering $\om \in \mathcal O_{B}$
is \textit{stationary} if the partial linear order defined by $\om$ on the set $E_{n}$ of all edges between levels $V_{n-1}$ and $V_{n}$,  does not depend on $n$ for $n>1$.
\end{definition}

It is well known that for every stationary ordered Bratteli diagram $(B, \om)$
one can define a `substitution $\tau$ read on $B$' by the following rule.
For each vertex $i \in V = \{1, 2,...,d\}$, we write $r^{-1}(i) = \{e_1,...,e_t\}$ where
$e_1 < e_2 < ... < e_t$ with respect to $\om$. Then we set
$\tau(i) = j_1j_2\cdots  j_t$ where $j_k = s(e_k),\ k =1,...,t$; this
defines the substitution read on $B$.
Conversely, such a substitution $\tau$ describes completely the stationary ordered Bratteli diagram $(B,\om)$ whose vertex set $V_n$ coincides with the alphabet of $\tau$ for all $n\geq 1$.

Now we give a useful description of infinite paths in an ordered
Bratteli diagram $(B,\om)$ (see also
\cite{bezuglyi_dooley_kwiatkowski:2006}).  Take $v\in V_n$ and
consider the finite set $E(v_0,v)$, whose cardinality is $h_v^{(n)}$. The
lexicographic ordering on $E(v_0,v)$ gives us an enumeration of its
elements from 0 to $h_v^{(n)} -1$, where 0 is assigned to the minimal path
and $h_v^{(n)} -1$ is assigned to the maximal path in $E(v_0,v)$. Note that  $h_v^{(1)} = f_{v v_0}^{(0)}$ for $ v\in V_1$, and we have by induction for $n>1$
\begin{equation*}\label{order}
h_v^{(n)} = \sum_{w\in s(r^{-1}(v))}|E(w,v)|h_w^{(n-1)},\ \ \ v\in V_n.
\end{equation*}

Let $y=(e_1,e_2,...)$ be an infinite path from $X_B$. Consider a sequence
$(P_n)$ of  enlarging finite paths defined by $y$ where $P_n = (e_1,...,e_n)
\in E(v_0, r(e_n))$,  $n\in \mathbb{N}$. Then every $P_n$ can be identified
with a pair $(i_n,v_n)$ where $v_n = r(e_n)$ and $i_n \in [0, h^{(n)}_{v_n}-1]$ is
the number assigned to $P_n$ in $E(v_0,v_n)$. Thus, every $y = (e_n)\in
X_B$ is uniquely represented as the infinite sequence $(i_n,v_n)$ with
$v_n =r(e_n)$ and $0\leq i_n \leq h^{(n)}_{v_n} -1$. We refer to the sequence
$(i_n,v_n)$ as the {\it associated sequence.}

\begin{proposition} Two infinite paths $e =(e_1,e_2,...)$ and $e' =
(e'_1,e'_2,...)$ from the path space $X_B$ are cofinal with respect to $\mathcal E$
if and only if the sequences $(i_n,v_n)$ and
$(i'_n,v'_n)$ associated to $e$ and $e'$ satisfy the condition: there exists
$m\in \N$ such that $v_n = v'_n$ and $i_n - i'_n = i_m - i'_m$ for all
$n\geq m$.
\end{proposition}

\begin{proof}  Suppose $e$ and $e'$ are cofinal. Take $m$ such that
$e_n=e_n'$ for all $n\geq m$. Consider the associated sequences
$(i_n,v_n)$ and $(i_n',v_n')$. Then we see that $v_n=v_n'$
for all $n\geq m$. Without loss of generality, we can assume that $c_m
=i_m-i_m'\geq 0$. This means that the finite path $P_m=P(e_1,...,e_m)$
is the $c_m$-th successor of the finite path $P'_m = P(e'_1,...,e'_m)$. Let
$c_{m+1}=i_{m+1}-i_{m+1}'$. By definition of the lexicographic ordering on
$E(v_0,v_{m+1})$, we obtain that $c_{m+1}=c_m$. Thus, by induction,
$c_{n}=c_m$ for all $n\geq m$.

Conversely, suppose that two associated sequences $(i_n,v_n)$ and
$(i_n',v_n')$ possess the property: there exists $m\in \N$ such that
$v_n=v_n'$ and $i_n-i_n'=i_m-i_m'$ for all $n\geq m$. To see that $e$
and $e'$ are cofinal, notice that $e_{m+1}$ and $e_{m+1}'$ are in
$E(v_m, v_{m+1})$. By definition of the lexicographic ordering on
$E(v_0,v_{m+1})$, we conclude that $e_{m+1}=e_{m+1}'$.
\end{proof}

\begin{proposition}
A Bratteli diagram $B$ admits an ordering $\om \in \mathcal O_{B}$ on
$B$ with $Int(X_{\max}(\om))\neq\emptyset$ if and only if there exist
$x =(x_i)\in X_B$ and $n>0$ such that  $U(x_1,\ldots,x_n)= \{y\in X_B :
y_i= x_i,\ i= 1,\ldots,n\}$  has no cofinal paths, i.e. $U(x_1,\ldots,x_n)$
meets each $\mathcal E$-orbit at most once. A similar result
holds for $Int(X_{\min}(\om))$.
\end{proposition}

\begin{proof} Let $x$ be an interior point of $X_{\max(\om)}$. Then there
is an  $n>0$ such that $U(x_1,\ldots,x_n) \subset X_{\max}(\om)$; thus,
$U(x_1,\ldots,x_n)$ contains no distinct cofinal paths.

Now, suppose that there exist $x =(x_i)\in X_B$ and $n>0$ such that
$U=U(x_1,\ldots,x_n)$ meets each $\mathcal E$-orbit at most once. Define
a linear order $\om_v$ on $r^{-1}(v),\ v\in V^*\setminus V_0$, as follows.
If there exists an $e\in r^{-1}(v)$ which is an edge in an infinite path $y\in U$,
then we order $r^{-1}(v)$ such that $e$ is maximal in $r^{-1}(v)$. If such an
$e$ does not exist, we order $r^{-1}(v)$ in an arbitrary way. It follows that
for this ordering $U\subset X_{\max}(\om)$.
\end{proof}

\begin{definition}A Bratteli diagram $B $ is called \textit{ regular} if for any ordering
$\omega \in \mathcal O_{B}$ the sets $X_{\max}(\omega)$ and
$X_{\min}(\omega)$ have empty interior.
\end{definition}
In particular, finite rank Bratteli
 diagrams are regular.

Given a Bratteli diagram $B$, we can  describe the set
of all orderings $\mathcal O_{B}$ in the following way.
Given a  vertex $v\in V^*\backslash V_0$, let   $P_v$ denote the set of all orders on $r^{-1}(v)$; an element in $P_v$
is denoted by $\om_v$.
Then $\mathcal O_{B}$ can be
represented as
\begin{equation}\label{orderings_set}
\mathcal O_{B} = \prod_{v\in V^*\backslash V_0}P_v .
\end{equation}
 Giving each set $P_v$ the discrete topology, it follows from (\ref{orderings_set}) that  $\mathcal O_{B}$ is a Cantor
set with respect to the product topology. In other words, two orderings
$\omega= (\om_v)$ and $\omega' = (\om'_v)$ from $\mathcal O_{B}$ are close
if and only if they agree on a sufficiently long initial segment: $\om_v = \om'_v,
v \in \bigcup_{i=0}^k V_i$.

It is worth  noticing that the order space $\mathcal O_B$ is sensitive with
respect to a telescoping. Indeed, let $B$ be a Bratteli diagram and $B'$
denote the diagram obtained
by telescoping of $B$ with respect to a subsequence $(n_k)$ of levels.
We see that any ordering $\omega$ on $B$ can be extended to the
(lexicographic) ordering $\om'$ on $B'$. Hence the map $L : \om \to \om'=L(\om)$
defines a closed proper subset $L(\mathcal O_B)$ of $\mathcal O_{B'}$.

The set of all orderings $\mathcal O_{B} $ on a Bratteli diagram $B $
can be considered also as a \textit{measure space} whose Borel structure is
generated by cylinder sets.  On the set
$\mathcal O_B $
we take the product measure $\mu = \prod_{v\in V^*\backslash V_0} \mu_v$ where  $\mu_v$
is a measure on the set $P_v$.   The case where each $\mu_v$
is the uniformly distributed measure on $P_v$ is of particular interest:
$\mu_v(\{i\}) =  (|r^{-1}(v)|!)^{-1}$ for every $i \in P_v$ and $v\in V^* \backslash V_0$.
Unless $|V_n|=1$ for almost all $n$,  if $B'$ is a telescoping of $B$, then in $\mathcal O_{B'}$, $L(\mathcal O_{B})$ is a set of zero measure.

%%%%%%%%%%%%%Vershik map%%%%%%%%
\subsection{Vershik maps}

\begin{definition}\label{VershikMap}
Let $(B, \omega)$ be an  ordered Bratteli
diagram. We say that $\varphi = \varphi_\omega : X_B\rightarrow X_B$
is a {\it (continuous) Vershik map} if it satisfies the following conditions:

(i) $\varphi$ is a homeomorphism of the Cantor set $X_B$;

(ii) $\varphi(X_{\max}(\omega))=X_{\min}(\omega)$;

(iii) if an infinite path $x=(x_1,x_2,\ldots)$ is not in $X_{\max}(\omega)$,
then $\varphi(x_1,x_2,\ldots)=(x_1^0,\ldots,x_{k-1}^0,\overline
{x_k},x_{k+1},x_{k+2},\ldots)$, where $k=\min\{n\geq 1 : x_n\mbox{
is not maximal}\}$, $\overline{x_k}$ is the successor of $x_k$ in
$r^{-1}(r(x_k))$, and $(x_1^0,\ldots,x_{k-1}^0)$ is the minimal path
in $E(v_0,s(\overline{x_k}))$.
\end{definition}

If $\om$ is an ordering on $B$, then one can always define the map
$\varphi_0$ that maps $X_B \setminus X_{\max}(\om)$ onto $X_B
\setminus X_{\min}(\om)$ according to (iii) of Definition
\ref{VershikMap}. The question about the existence of the Vershik map is
equivalent to that of an extension of  $\varphi_0 : X_B \setminus
X_{\max}(\om) \to X_B \setminus X_{\min}(\om)$ to a homeomorphism
of the entire set $X_B$.  If $\om$ is a proper ordering, then
$\varphi_\om$ is a homeomorphism. For a finite rank Bratteli diagram $B$, the
situation is simpler than for a general Bratteli diagram because the sets
$X_{\max}(\om)$ and $X_{\min}(\om)$ are finite.
\medskip

\begin{definition}\label{Good_and_bad}  Let $B$ be a Bratteli diagram
$B$.  We say that an ordering $\om\in \mathcal O_{B}$ is \textit{perfect}
if $\om$ admits a Vershik map $\varphi_{\om}$ on $X_B$.
Denote by  $\mathcal P_B $ the set of
all perfect orderings  on $B$.  We call an ordering belonging to $\mathcal P_B^{\, c}$ (the complement of
$\mathcal P_B$ in $\mathcal O_{B}$) \textit{imperfect}.
\end{definition}

We observe that for a regular Bratteli diagram with an ordering $\om$, the Vershik map $\varphi_\om$, if it exists, is
defined in a unique way. More precisely, if $B$ is a
regular Bratteli diagram such that the set $\mathcal P_B$ is not
empty, then the map $\Phi: \omega \mapsto \varphi_\omega: \mathcal
P_B\to Homeo(X_B)$ is injective. Also, a necessary condition for $\om\in
\mathcal P_{B}$ is that $|X_{\max}(\om)|=|X_{\min}(\om)|$.

\begin{remark}
 We  note that if $B$ is a simple Bratteli diagram, with positive entries in all its incidence matrices, then the set
 $\mathcal P_B\neq \emptyset$. Indeed, it is not hard to see that if  $x$ and
$y$ are two paths in $X_B$ going through disjoint edges at each level, then one can find an ordering $\om$ on $B$ such that $X_{\max}(\om)
= \{x\}$ and $X_{\min}(\om) = \{y\}$: simply choose all maximal edges in $E_n$ to go through the same vertex that $x$ goes through at level $n-1$, and all minimal edges in $E_n$ to go through the same vertex that $y$ goes through at level $n-1$, for each $n$. Then  $\om$ is properly
 ordered, and  so $\om\in \mathcal P_{B}$.

Another example of a family of  perfect (indeed proper) orders  for a simple
Bratteli diagram, all of whose incidence matrices are positive, is the following.
For each $n$, fix a labeling $V_n=\{v(n,1), \ldots v(n,k_n)\}$ of $V_n$. Take  $v\in V_{n+1}$, and 
enumerate the edges from $E(v(n,1), v)$ in an arbitrary order from $0$ to $|E(v(n,1), v)| -1$. Similarly, for $2\leq i \leq k_n$, we enumerate edges from $E(v(n,i), v)$ by numbers from $\sum_{j=1}^{i-1} |E(v(n,j), v)|$ to $\sum_{j=1}^{i}       |E(v(n,j), v)|-1$. Repeating this procedure for each vertex $v\in V^*\backslash V_0$ and each level $n$, we define an order $\om_0$ on $B$, called a natural order. This is a variation of the well known `left-to-right' order. For $\om_0$, the unique minimal path runs through $v(n,1)$, and  the unique maximal path  runs through $v(n,k_n)$.
\end{remark}

In the next section, we will describe a class of non-simple Bratteli diagrams that
do not admit a perfect ordering.

\begin{proposition} \label{Odometer}
Let $B$ be a simple Bratteli diagram, where the entries of the incidence matrices $(F_n)$ are positive. Then
$\mathcal P_B =\mathcal  O_B$ holds if and only if
 $B$ is rank 1.
\end{proposition}

\begin{proof}
The part `if' is obvious because the condition  $|V_n| =1$ for infinitely many
levels  $n$ implies any ordering is proper.

Conversely, suppose that the rank of $B$ is at least 2.  Then, for some $N$, $|V_n|\geq 2$ when $n>N$. We need to show that, in this case, there are imperfect orderings.

First,  assume that  infinitely often, $|V_n| \geq 3$. Call  three distinct vertices at these levels $u_n, \, v_n$ and $w_n$. For the other levels $n>ÊM$, there are at least two distinct vertices $u_n$ and $v_n$.
 For levels $n$ such that $|V_n| \geq 3$, choose all maximal edges in $E_{n+1}$ to have source $w_n$. Let the minimal edge with range $u_{n+1}$, $v_{n+1}$ have source $u_n$, $v_n$ respectively.  For levels $n$ such that $|V_n|=2$,  let the minimal edge with range $u_{n+1}$, $v_{n+1}$ have source $u_n$, $v_n$ respectively.
  Any order which satisfies these constraints has only one maximal path, and at least two minimal paths, so cannot be perfect.

Next suppose that $B$ has rank 2, and
suppose  two sequences of vertices $(v_n)$ and $(w_n)$ can be found such that $v_n \neq w_n$ for each $n>N$, $v_n, \, w_n \in V_n$ and
 $|E(w_n, w_{n+1})| >1$ infinitely often.
Let the minimal edge with range $v_{n+1}$ have source $v_n$. Similarly, let the minimal edge with range $w_{n+1}$ have source $w_n$. Whenever  $| E(w_n, w_{n+1})   |>1$, choose all maximal edges in $E_{n+1}$  to have source $w_n$.  The resulting order has one maximal and two minimal paths.

Finally suppose that $B$ does not satisfy the above conditions. Then, for all large $n$, the matrices
$F_n= \left(
      \begin{array}{cc}
        1 & 1 \\
        1  &  1 \\
            \end{array}
    \right),$ and there are orders on $B$ with two maximal and two minimal paths. To see this we just ensure that for all large $n$, the two minimal edges have distinct sources, as do the two maximal edges.  Now Example \ref{morse_example} shows that no such ordering is perfect.

  \end{proof}

In contrast, one can find
aperiodic diagrams for which any ordering is perfect. Indeed, it suffices to take a
rooted tree and turn it into a non-simple  Bratteli diagram $B$ by replacing
every single edge with a strictly larger number of edges. Then
every ordering on $B$ produces a continuous Vershik map.

\begin{remark} Let $(B, \om)$ be an ordered Bratteli diagram and let $\om'$
be an ordering on $B$ such that $\om$ and $\om'$ are different on $r^{-1}(v)$
only for a finite number of vertices $v$. Then $\om $ is perfect if and only if
$\om'$ is perfect.
\end{remark}

\begin{proposition}\label{continuity} Let $B$ be a regular Bratteli diagram
such that the set $\mathcal P_B$ is not empty. Let $\mathcal P_B$ be
equipped with the   topology induced from $\mathcal O_B$ and let
 the set $\Phi(\mathcal P_B)$ be equipped with the topology of uniform convergence induced  from the group $Homeo(X_B)$  where the map $\Phi: \omega \mapsto \varphi_\omega $ has been defined above.  Then  $\Phi : \mathcal P_B \to \Phi(\mathcal P_B)$ is a homeomorphism.
\end{proposition}

\begin{proof} We need only to show that $\Phi$ and $\Phi^{-1}$
are continuous because injectivity of $\Phi$ is obvious.

Fix an ordering $\omega_0 \in \mathcal P_B$ and let $\varphi_{\omega_0}$
be the corresponding Vershik map. Consider a neighborhood $W=
W(\varphi_{\omega_0} ; E_1,...,E_k) = \{f\in Homeo(X_B) : f(E_i) =
\varphi_{\omega_0}(E_i), \ i=1,...,k\}$ of $\varphi_{\omega_0}$ defined by
clopen sets $E_1,...,E_k$. It is well know that the uniform topology is generated by the base of neighborhoods $\{W\}$. Take $m\in \N$ such that all clopen sets
$E_1,...,E_k$ `can be seen' at the first $m$ levels of the diagram $B$.
This means that every set $E_i$ is a finite union of the cylinder sets
defined by finite paths of length  $m$.

Suppose $\omega_n \to \omega_0$ where $\omega_n \in \mathcal P_B$.
By (\ref{orderings_set}), the ordering $\omega_0$ is an infinite sequence
in the product $\prod_{v\in V^* \backslash{V_0}} P_v$. Let $Q$ be the neighborhood of $\omega_0$ in $\mathcal O_{B}$ which is defined by the finite part of
$\omega_0$ from $v_0$ to $V_{m+1}$. Find $N$ such that $\omega_n \in Q$
 for all $n\geq N$. This means that the ordering $\omega_n\ (n\geq N) $
 agrees with $\omega_0$ on the first $m+1$ levels of the diagram $B$.
 Therefore, $\varphi_{\omega_n}$ acts as $\varphi_{\omega_0}$ on all finite
 paths from $v_0$ to $V_m$. Hence, $\varphi_{\omega_n}(E_i) =
 \varphi_{\omega_0}(E_i)$ and $\varphi_{\omega_n}\in W$.

Conversely, let $\varphi_{\omega_n} \to \varphi_\omega$ in the topology
of uniform convergence; we prove that $\omega_n \to \omega$. Take the
neighborhood $Q(\omega)$ of $\om$  consisting of all orderings $\omega'$ such
that $\omega'$ agrees with $\omega$ on the sets $r^{-1}(v)$, where
$v\in \bigcup_{i=1}^N V_i$. Let $F_1,...,F_p$ denote all cylinder subsets of
$X_B$ corresponding to the finite paths between $v_0$ and the vertices
from $\bigcup_{i=1}^{N+1} V_i$. Consider the neighborhood $W =
W(\varphi_\omega; F_1,...,F_p)$. Then there exists an $m\in \N$ such that
$\varphi_{\omega_i} \in W$ for $i\geq m$. This means that
$\varphi_{\omega_i}(F_j) = \varphi_{\omega}(F_j)$ for all $j=1,...,p$.
Let us check that $\omega_i\in Q(\omega)$ for $i\geq m$. Indeed, if one
assumes that $\omega'\notin Q(\omega)$ then there exists a least $k$
and a vertex $v\in V_k$ such that $\omega$ and $\omega'$ define
different linear orders on $r^{-1}(v)$, but $\omega$ and $\omega'$ agree
for all $v\in \bigcup_{i=1}^{k-1} V_i$. Let $e$ be an edge from $r^{-1}(v)$
such that the $\omega$-successor and $\omega'$-successor of $e$ are
different edges. Then take the cylinder set $F$ which corresponds to the finite
path $(f,e)$, where $f$ is  the maximal path from $v_0$ to $s(e)$ for
 both the  orders. It follows from the above construction that
$\varphi_\omega(F) \neq \varphi_{\omega'}(F)$, a contradiction.
\end{proof}

\begin{theorem}\label{GoodBadOrderingsAreDense}
Let $B$ be a simple rank $d$ Bratteli diagram where $d\geq 2$, and all incidence matrix entries are positive.
 Then both sets $\mathcal P_B$ and ${\mathcal P}_B^{\, c} $ are dense in $\mathcal O_B$.
\end{theorem}

\begin{proof} By Proposition
\ref{Odometer},  $\mathcal P_B^{\, c} \neq \emptyset$. Take an ordering $\om \in \mathcal O_B$ and consider its
neighborhood $U_N(\om) = \{\omega' \in \mathcal O_B : \om\ \mbox{and}
\ \omega' \ \mbox{coincide\ on\ $r^{-1}(v)$\ for \ all}\ v \in
\bigcup_{i=1}^N V_i\}$. We have assumed that $N$ is large enough such that $|V_n|\geq 2$ for $n>N$.

Then there
exists a perfect ordering $\om_1$  belonging to $U_N(\om)$. To see this, choose $(u_n)_{n>N}$, $(v_n)_{n>N}$ where $u_n\neq v_n$ and $u_n$, $v_n \in V_n$.  Choose an ordering all of whose maximal edges in $E_{n+1}$ have source $u_n$ and all of whose minimal edges in $E_{n+1}$ have source $v_n$, for $n>N$. Let this ordering  agree with $\om$ up to level $N$. This ordering is proper, hence perfect.

Conversely, if $\om$ is perfect, we can  construct $\om^{N}$ by letting
 $\om^{N}$ agree with $\om$ on the first $N$ levels. Beyond level $N$, we work
 as in the proof of Proposition
\ref{Odometer} to define $\om^N$ so that it is imperfect.

\end{proof}

%%%%%%%%%%%FINITE RANK DIAGRAMS%%%%%%%
\section{Finite rank ordered Bratteli diagrams}\label{language_skeleton}

 In this section, we focus on the study of orderings on a finite rank
Bratteli diagram $B$. To do this, we define new notions related to an unordered finite
rank Bratteli diagram that will be used in our considerations.
 If  $(B, \om)$ is ordered, and $V_n= V$ for each $n$, in Section \ref{language_section} we first define the {\em language} generated by $\om$, and characterize whether $(B,\om)$ is perfect in terms of the language of $\om$.  Our notions of skeleton and associated graph are defined in Section \ref{skeleton_definition} for  non-ordered diagrams. We note  that on one diagram, there exist several skeletons. By  telescoping a perfectly ordered diagram in a particular way, we will obtain the (unique, up to labeling) {\em skeleton} associated to the lexicographical image of $\om$ under the telescoping.  In the {\em associated graph} $\mathcal H$, paths  will correspond to (families of) words in $\om$'s language. Given a skeleton $\mathcal F$ on a diagram,  we describe how  $\mathcal H$ constrains us when trying to extend $\mathcal F$
 to a perfect order.

 In Section \ref{no_perfection} we describe a class of non-simple diagrams that do not admit any perfect ordering, using the poor connectivity properties of any skeleton's associated graph.  In Section \ref{odometer}   we give descriptions of perfect orderings that yield odometers, in terms of their language, and explicitly
describe, in terms of an associated skeleton and associated graph,  the class of rank $d$ diagrams  that can have a
perfect ordering with exactly $k\leq d$ maximal and minimal paths.

%%%%%%%%%%%%%%%%LANGUAGE%%%%%%%%%

\subsection{Language of a finite rank diagram}\label{language_section}

Let $\om$ be an ordering on a  Bratteli diagram $B$, where
$V_{n}= V $ for each $n\geq 1$, and $|V|=d$.
For each vertex $v\in V_n$ and each $m$ such that $1\leq m<n$, consider
$\bigcup_{w\in V_{m}}E(w,v)$ as the $\om$-ordered set
$\{e_{1},\ldots e_{p}\}$ where $e_{i}<e_{i+1}$ for $1\leq i\leq
p-1$. Define the word $w(v,m,n):=s(e_{1})s(e_{2})\ldots s(e_{p})$ over
the alphabet $V$.
We use the notation $w' \subseteq w$ to indicate that $w'$ is a subword
of $w$, and, if $w$ and $w'$ are two words, by $w w'$ we mean the word which is the concatenation of $w$ and $w'$.

\begin{definition}\label{language}
The set
$$
\mathcal L_{B, \om}= \{w: w \subseteq w(v_{n},m_{n},n), \ \mbox{ for\ infinitely\ many}\ n \ \mbox{where}\  v_{n}\in V_n,  1\leq m_{n}<n \}
$$
is called the \textit{language} of $B$ with respect to the ordering $\om$.
\end{definition}

We remark that the notion of the language $\mathcal L_{B, \om}$ is not always  robust under telescoping:  let $(B', \omega')$ be a telescoping of an ordered Bratteli diagram $(B, \omega)$ where
$\omega' = L(\om)$. Then $\mathcal L_{B', \om'} \subset \mathcal L_{B, \om}$ where the inclusion can be strict. For example, consider $B$ where
\begin{equation}
F_{2n} = \left(
      \begin{array}{cc}
        1 & 2 \\
        2     &  2 \\
      \end{array}
    \right), \ \ \
F_{2n-1} = \left(
      \begin{array}{cc}
        2 & 1 \\
        3  &  1 \\
      \end{array}
    \right), \ \ n \geq 1.
\end{equation}
Let $\om$ be defined by the substitution
$\tau_{1}(a)=aba$, $\tau_{1}(b) =aaba$ on $E_{2n}$, and by the substitution
$\tau_{2}(a) = bab$, $\tau_{2}(b)=abba$ on $E_{2n-1}$ for $n\geq 1$. Thus  the order of letters in a word $\tau(v)$ determines the order on the sets of edges with  range $v$. Then
$\{aa, ab, ba, bb\}\subset \mathcal L_{B, \om}$. Now  telescope $B$ to the levels $(2n+1)$ to get the stationary Bratteli diagram $B'$ whose incidence matrix is
\begin{equation}
F_{n}' =
 \left(
      \begin{array}{cc}
        1 & 2 \\
        2  &  2 \\
      \end{array}
    \right)
\cdot
\left(
      \begin{array}{cc}
        2 & 1 \\
        3  &  1 \\
      \end{array}
    \right)
=
\left(
      \begin{array}{cc}
        8 & 3 \\
        10  &  4 \\
      \end{array}
    \right)
\end{equation}
for each $n\geq 1$, so that
$\om':=L(\om)$ is defined by the substitution $\tau:=\tau_{1}\circ
\tau_{2}$ where $\tau(a) = aaba\,aba\,aaba$ and $\tau(b) =
aba\,aaba\,aaba\,aba$, then $bb \,\not\in \mathcal L_{B', \om'}.$
Note however  that both $\om$ and $\om'$ are perfect (in fact proper).

Also, in the special case where $B$ is stationary and $\om$ is defined by a
substitution $\tau$ (so that $\om$ is also stationary), we see that $\mathcal
L_{B, \om}$ is precisely the language ${\mathcal L}_{\tau}$ defined by the
substitution $\tau$, and in this case, if $B'$ is a telescoping of $B$
to levels $(n_{k})$ with $\om'=L(\om)$, then $\mathcal L_{B, \om}=
\mathcal L_{B', \om'}$. Indeed, any word $w \in \mathcal L_{B, \om}$ is
a subword of $\tau^{j}(a)$ for some $j \in {\mathbb N}$ and letter
$a$. Now the order on the $k$-th level of $B'$ is generated by
$\tau^{n_{k}-n_{k-1}}$, and as long as $n_{k}-n_{k-1}>j$, we will see $w$
as a subword of $w(a,n_{k-1},n_{k})\subset\mathcal L_{B', \om'}$. The
relationship between ${\mathcal L}_{B,\om}$ and the continuity of the
Vershik map has been studied in \cite{yassawi:closing:2011} in the case
where $\om$ is stationary, i.e., generated by a substitution, and also
in \cite{holton:2001}\footnote{The relevant formula on Page 5 is
incorrect in the final version: the correct version is in the preprint
which can be found at http://combinatorics.cis.strath.ac.uk/papers/lucaz.}.

 \begin{definition}Suppose $B$ is such that $V_n=V$ for each $n\geq 1$.
If $\om$ is an order on $B$, where a maximal
(minimal) path $M$ ($m$) goes through the same vertex $v_{M}$
($v_{m}$) for each  level $n\geq 1$ of $B$, we will call this path {\em
vertical}.
\end{definition}

We note that for any order $\om$ on a finite rank Bratteli diagram $B$ there exists a telescoping $B'$ of $B$ such that the extremal (maximal and minimal) paths with respect to $\om' = L(\om)$ are vertical.

The following proposition characterizes when $\om$ is
a perfect ordering on such a finite rank Bratteli diagram.

\begin{proposition} \label{ExistenceVershikMap}
Let $(B, \om)$ be a finite rank ordered Bratteli diagram, where
$V_{n}= V$ for each $n\geq 1$ Suppose that the
  $\om$-maximal  and $\om$-minimal paths  $M_1,...,M_k$ and $m_1,...,m_{k'}$  are vertical passing through the vertices $v_{M_1},\ldots , v_{M_k}$ and $v_{m_1},\ldots , v_{m_{k'}}$ respectively.  Then $\om$ is perfect if and only if

\begin{enumerate}\item $ k = k'$ and

\item there is a permutation $\sigma$ of $\{1,\ldots k\}$ such that for
each $i\in \{1,...,k\}$, $v_{M_i}v_{m_{j}} \, \in \mathcal
L_{B,\om}$ if and only if $j=\sigma(i)$.

\end{enumerate}

\end{proposition}

\begin{proof} We first assume that the Vershik map $\varphi_\om$ exists.
Then $\varphi_\om$ defines a bijection between the finite sets
$X_{\max}(\om)$ and $X_{\min}(\om)$ by sending each $M_i$ to some $m_j$: let $\sigma(i)=j$.
Clearly, $k = k'$. We need to check that $v_{M_i}v_{m_j}$ is in the
language ${\mathcal L}_{B,\om}$  if and only if $j=\sigma(i)$.
 It follows from continuity of $\varphi_\om$ and the relation
$\varphi_\om(M_i) = m_j$ that if $x_n \to M_i$ then $\varphi_\om(x_n) = y_n
\to m_j$ as $n\to \infty$. We see that, for
every $n$, the condition $\varphi_\om(x_n) = y_n$ implies that  $v_{M_i}v_{m_j} \,\in
w(v,m,N)$ for some $v \in V_{N}$ and some $m<N$,  because $x_n$ and $y_n$ are taken from neighborhoods
generated by finite paths going through $v_{M_i}$ and $v_{m_j}$
respectively.  Furthermore, as $n\rightarrow \infty$, so does $N$ and also $m$. Hence
$v_{M_i}v_{m_j}\,\in {\mathcal L}_{B,\om}$ when $j=\sigma(i)$.
By the same argument,  if $v_{M_i}v_{m_k}\in
\mathcal L_{B,\om}$ for some $k\neq \sigma(i)$, then  one can find $x_n \rightarrow M_i$ such that $\varphi_\om(x_n) = y_n \rightarrow m_k$, a contradiction.

Conversely, assuming that (1) and (2) hold, extend $\varphi_{\om}$ to
 $X_{\max}(\om)$ by defining $\varphi(M_{i}) := m_{\sigma(i)}$.
It is obvious that
 $\varphi_\om$ is one-to-one. Fix a pair $(M_i, m_j)$ where $j = \sigma (i)$, and let $x_n \to M_i$
as $n\to \infty$; we show that $y_n =  \varphi_\om(x_n) \to m_j$.
We can assume that the first $n$ edges of $x_n$
coincide with those of $M_i$, i.e. $x_n = \overline{e}_{\max}^{(n)}(v_0, v_{M_i})e_{n+1}e_{n+2}\cdots$
where $e_{n+1}$ is not maximal in $r^{-1}(r(e_{n+1}))$. Then
$y_n = \overline{f}_{\min}^{(n)}(v_0, s(e'_{n+1}))e'_{n+1}
e_{n+2}\cdots$ where $e'_{n+1}$ is the successor of $e_{n+1}$. Take a
subsequence $(y'_n)$ of $(y_n)$ convergent to a point $z\in X_B$.
By construction, $z$ must be a minimal path. It follows from the uniqueness
of $j$ in condition (2) that $z = m_j$; this proves the continuity of
$\varphi_\om$.

\end{proof}

\begin{example}
 \label{stationary_example_a} Let $(B, \om)$ be a stationary ordered Bratteli diagram whose vertex set $V_n= \{a,b,c,d\}$ for each $n\geq 1$, and where the ordering is defined  by the substitution $a \to acbda, b\to
bdcbdacb, c\to acdcb, d\to bdacda$. There are two pairs of vertical
maximal and minimal paths going through vertices $a$ and $b$.
The words of length two that appear in $\mathcal L_{B, \om}$ are $\{aa, ac,bb,  bd, cb, cd, da, dc  \} $ and using Proposition \ref{ExistenceVershikMap}, we conclude that $\om \in \mathcal P_{B}$ and $\varphi_\om(M_a) = m_a$, and  $\varphi_\om(M_b) = m_b$.
\end{example}

\begin{example}\label{morse_example}

Let $B$ be the stationary ordered Bratteli diagram whose vertex set $V_n= \{a,b\}$ for each $n\geq 1$, and whose incidence matrices $F_n =  \left(
      \begin{array}{cc}
        1 & 1 \\
        1  &  1 \\
            \end{array}
    \right)$ for each $n$. We claim that any ordering on $B$ with two maximal and two minimal paths cannot be perfect. The only possible choices to ensure that $\om$ has this many extremal paths is, for all large $n$, to either choose the ordering
    $w(a,n,n+1) = ab$ and  $w(b,n,n+1) = ba$,  or to choose  the ordering  $w(a,n,n+1) = ba$ and  $w(b,n,n+1) = ab$. Whatever choice one makes at level $n$ and level $n+1$, all four words $\{ aa,ab,ba,bb\}$ occur somewhere in one of the two words $w(a,n,n+2)$ or $w(b,n,n+2)$. Thus, $\om$ cannot be perfect.

\end{example}

\begin{remark} \label{telescoping_assumption} Suppose that $(B,\om)$ satisfies the conditions of Proposition \ref{ExistenceVershikMap}. This means that there exists an $N$ such that if we see
$v_{M_i}v_{m_j}$ appearing in some word $w(v,m,n)$ with $m\geq N$, then $j=\sigma(i)$. We can telescope $B$ to levels $N, N+1, N+2, \ldots $ so that if we see $v_{M_i}v_{m_j}$ appearing in some word $w(v,m,n)$ with $m\geq 1$, then $j=\sigma(i)$. Thus, unless otherwise indicated, for the remainder of Section \ref{language_skeleton}, when we have an ordered diagram $(B,\om)$ that satisfies the conditions of Proposition \ref{ExistenceVershikMap}, {\em we shall assume that  if $v_{M_i}v_{m_j} \subset w(v,m,n)$ with $m\geq 1$, then $j=\sigma(i)$.}
\end{remark}

We now generalize Proposition \ref{ExistenceVershikMap} to arbitrary finite rank diagrams, where the extremal paths are not necessarily vertical. Although the notion of language is not defined for these diagrams, we can still define, and use words $w(v,m,n)$ for $v\in V_n$ and $m<n$. The proof of this lemma is elementary, so we  omit it,
although Figure \ref{continuity_1} is explanatory.

\begin{lemma}\label{ExistenceVershikMap_Part_2}
Let $B$ be a finite rank diagram.
  Then the following statements are equivalent:

 \begin{enumerate} \item

$\om\not\in \mathcal P_{B}$;

\item For some $\om$ maximal path $M$, and two $\om$ minimal paths $m$ and $m^{*}$, there exist strictly increasing sequences of levels   $(n_{k})$, $(n_{k}^{*})$, $(N_{k})$ and  $(N_{k}^{*})$, vertices $\{w_{k}, \, v_{k}\}\subset V_{n_k}$,
$\{w_{k}^{*},\, v_{k}^{*}\}\subset V_{n_{k}^{*}}$, vertices $u_{k} \, \in  V_{N_{k}}$,  $u_{k}^{*} \, \in  V_{N_{k}^{*}}$ such that $M$  passes through $w_{k}$ and $w_{k}^{*}$, $m$ and $m^{*}$ pass through $v_{k}$ and $v_{k}^{*}$ respectively,   and $w_{k}v_{k} \subset w(u_{k},n_{k},N_{k})$,
 $w_{k}^{*}v_{k}^{*} \subset w(u_{k}^{*},n_{k}^{*},N_{k}^{*})$.
\end{enumerate}\end{lemma}

\begin{figure}[h]
\centerline{\includegraphics[scale=0.9]{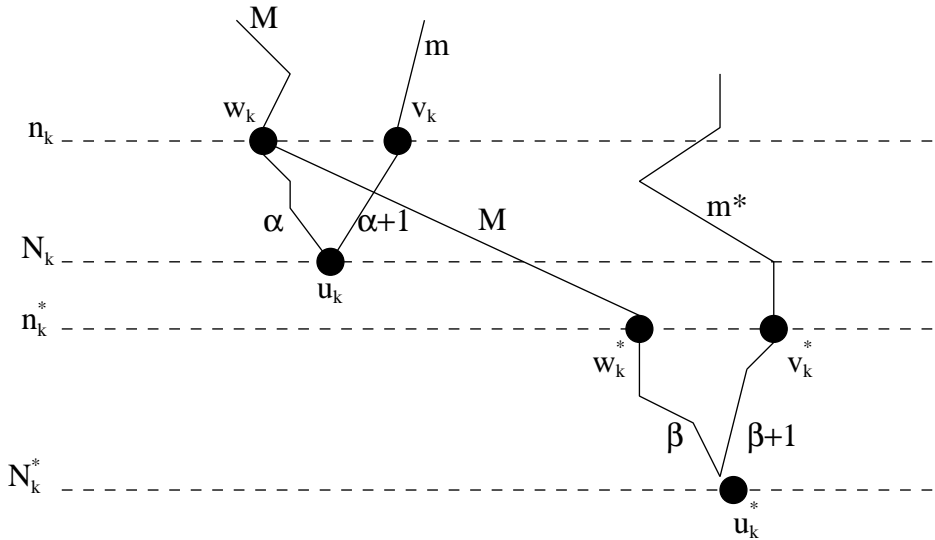}}
\caption{A discontinuous $\varphi_\om$.}
  \label{continuity_1}
\end{figure}

\begin{lemma}\label{good_orders_and_telescoping}
Let $B$ be a Bratteli diagram of finite rank and $B'$ a telescoping of
 $B$. Then an ordering $\omega \in \mathcal P_{B}$ if and only if the
 corresponding lexicographic ordering $\om' = L(\om)\in \mathcal P_{B'}$.
\end{lemma}

\begin{proof}
If $\om$ does not determine a Vershik map, then by Lemma \ref{ExistenceVershikMap_Part_2}, there is a maximal path $M$, two
distinct minimal paths $m$ and $m^{*}$, infinite sequences of levels
$(n_{k})$ and $(n_{k}^{*})$, $(N_{k})$ and  $(N_{k}^{*})$, vertices $\{w_{k}, \, v_{k}\}\subset V_{n_k}$, $\{w_{k}^{*},\, v_{k}^{*}\}\subset V_{n_{k}^{*}}$ and vertices $u_{k} \, \in
V_{N_{k}}$, $u_{k}^{*} \, \in V_{N_{k}^{*}}$ such that $M$ passes through
$w_{k}$ and $w_{k}^{*}$, $m$ ($m^{*}$) pass through $v_{k}$ ($v_{k}^{*}$), and $w_{k}v_{k}
\subset w(u_{k},n_{k},N_{k})$, $w_{k}^{*}v_{k}^{*} \subset
w(u_{k}^{*},n_{k}^{*},N_{k}^{*})$ (see Figure \ref{continuity_1}).  Note that in $B$, it cannot be the case
that for infinitely many levels, the minimal paths go through the same
vertex - otherwise they are not distinct. Thus, there is some $N$ such
that if $n\geq N$, the level $n$ edge in $m$ has a different source
and range from the level $n$ edge in $m^{*}$.

Let  $B'$ be   a telescoping of $B$ to levels $(m_k)$.   If the images of $M$, $m$ and $m^{*}$ in $B'$ are denoted by $M'$, $m'$ and $({m}^{*})'$ respectively, then by the comment above, apart from an initial segment,  the paths $m'$ and $({m}^{*})'$ pass through distinct vertices in $B'$.

Find the levels $m_{j}$ and $m_{J}$ in $(m_k)$ such that $m_{j-1}< n_{k} \leq m_{j}$, $m_{J-1}< N_{k} \leq m_{J}$, and let
  $E_{j}'$ denote the edge set in $B'$ obtained by telescoping  between  $m_{j-1}$-st and $m_{j}$-th levels of $B$, and  let $E_{J}'$ denote the edge set obtained by
telescoping between the $m_{J-1}$-st and $m_{J}$-th levels of $B$.
Let the path $M$ go through $w_{j}'\,\in V_{m_{j}}$, and $m$ through
$v_{j}'\,\in V_{m_{j}}$.

 Let $u_{J}'\, \in V_{m_{J}}$ be any
vertex such that there is a path from $u_{k}\,\in V_{N_{k}}$ to
$u_J'$. Then for the corresponding vertices $w_{j-1}',\,
v_{j-1}' \, \in V_{j-1}'$ and $u_{J}'\,\in V_{J}'$
respectively it is the case that $w_{j-1}'v_{j-1}'\,\in
w(u_{J}',j-1,J)$, with $M'$ passing through $w_{j-1}'$, and
$m'$ passing through $v_{j-1}'$.
Repeat this procedure
for $m^{*}$. 
By Lemma \ref{ExistenceVershikMap_Part_2}, the ordering $\om'$ on $B'$, obtained  from $\om$ by telescoping, does not determine a Vershik map.

The converse is proved similarly.
%Conversely, suppose that $\om'$ does not determine a Vershik map.
%Then Lemma \ref{ExistenceVershikMap_Part_2} applies to
%$\om'$, yielding the sequences and vertices with the specified
%properties; via the process of splitting $B'$ to get back $B$,
%these sequences and vertices in $B'$ are mapped to sequences and
%vertices in $B$ with the properties required by  Lemma
%\ref{ExistenceVershikMap_Part_2}, implying that $\om\in \mathcal P^{c}$.
\end{proof}

Lemma \ref{good_orders_and_telescoping} and the compactness of $X_B$ imply the following corollary.

\begin{corollary}
  Suppose that $B$ has rank $d$. Then
$\om\in \mathcal P_{B}$ if and only if
 there exists a telescoping $(B',\om')$ of $(B,\om)$ such that $V_{n}' = V'$ for each $n\geq 1$,  the
  $\om'$-maximal  and $\om'$-minimal paths  $M_1,...,M_k$ and $m_1,...,m_{k'}$  are vertical,  and $\om'$ satisfies
the conditions of Proposition \ref{ExistenceVershikMap}.

\end{corollary}

Now we give another criterion which guarantees the existence of Vershik map on
an ordered Bratteli diagram $(B,\om)$ (not necessarily of finite
rank). Let $\om = (\om_v)_{v\in V^* \backslash V_0}$ be an ordering on a regular Bratteli diagram
$B $. For every  $x_{\max} = (x_n) \in X_{\max}(\om)$, we define
the set  $Succ(x_{\max}) \subset X_{\min}(\om)$ as follows: $y_{\min} =
(y_n)$ belongs to the set $Succ(x_{\max})$ if for infinitely many $n$
there exist edges $y' \in s^{-1}(r(x_n))$ and $y'' \in s^{-1}(r(y_n))$
such that $r(y') = r(y'') = v_{n+1}$ and  $y''$ is the successor of $y'$
in the set $r^{-1}(v_{n+1})$. Given a path $y_{\min} \in X_{\min}(\om)$,
we define the set $Pred(y_{\min})\subset X_{\max}(\om)$ in a similar way.
It is not hard to prove that the sets  $Succ(x_{\max})$
and $Pred(y_{\min})$ are non-empty and closed for any $x_{\max}$ and
$y_{\min}$.

\begin{proposition}\label{existenceVershikMap}
 An ordering $\om = (\om_v)_{v\in V^*\backslash V_0}$ on a regular Bratteli diagram $B$ is
perfect if and only if for every $x_{\max} \in X_{\max}
(\om)$ and $y_{\min} \in X_{\min}(\om)$ the sets $Succ(x_{\max})$ and
$Pred(y_{\min})$ are singletons.
\end{proposition}

\begin{proof}  Let $x_{\max}$ be any path from $X_{\max}(\om)$.  If  $Succ(x_{\max})= \{y_{\min}\}$,
then one can define $\varphi_{\om} : x_{\max} \to y_{\min}$. Since $Pred(y_{\min})$ is also a singleton, we obtain a one-to-one correspondence between the sets of maximal and minimal paths.
The fact that $\varphi_{\om}$ is continuous can be checked directly.

Conversely, if $\om$ is perfect, then it follows from the existence of the Vershik map $\varphi_{\om}$ that
either of the sets $Succ(x_{\max})$ and $Pred(y_{\min})$ must be singletons.
\end{proof}

%%%%%%%%%%%%
%%%%%%%%%%%%%%Skeleton and associated graphs%%%%%%%%
%%%%%%%%%%%

\subsection{Skeletons and associated graphs}\label{skeleton_definition}
Let $B$ be a finite rank Bratteli diagram. We do not need to assume here that $B$ is simple unless
we state this explicitly.
If $\om$ is an order on $B$, and $v\in V^*\backslash V_0  $, we denote the  minimal edge with range $v$ by  $\ol e_v$ , and  we denote the maximal edge with range $v$ by $\wt e_v$.

\begin{lemma}\label{well_telescoping_lemma}
Let $(B',\om')$ be a  rank $d$ ordered diagram. Then there exists a telescoping $(B,\om)$ of $(B',\om')$ such that
\begin{enumerate}
\item $|r^{-1}(v)| \geq 2$ for each $v\in V^*\backslash V_0 $,
\item $V_n=V$ for each $n\geq 1$ and $|V|=d$,
\item all $\om$-extremal paths are vertical, with $\wt V$,  $\ol V$ denoting the sets of vertices through which maximal and minimal paths run respectively, and
\item $s(\wt e_v)\in \wt V$ and $s(\ol e_v) \in \ol V$ for each $v\in V^*\backslash (V_0 \cup V_1 )$, and this is independent of $n$.
\end{enumerate}

In addition, if $\om \in \mathcal P_B$, we can further telescope so that

\begin{enumerate}
\setcounter{enumi}{4}
\item if $\wt v\ol v $ appears as a subword of some $w(v,m,n)$ with $m\geq 1$, then , then $\sigma(\wt v) = \ol v$ defines a one-to-one correspondence between the sets $\wt V$ and $\ol V$.
\end{enumerate}
\end{lemma}

\begin{proof} Property (1) is guaranteed by Lemma \ref{aperiodic-property}. To obtain property (2), we telescope through the levels $(n_k)$ such that  $|V_{n_k}|=d$, where $d$ is the rank of $B'$.  To obtain (3), note that each maximal path $M'$ passes through one vertex $\wt v_M$ infinitely often. Telescope $B$ to  the levels where this occurs; the image $M$ of $M'$ is then a maximal vertical path passing though $\wt v_M$ at each level. Repeat this procedure for each maximal path $M'$ and each minimal path $m'$. To see (4), we assume we have telescoped so that properties (1) - (3) hold. We denote the vertical maximal path passing through $\wt v\in \wt V$ by
$M_{\wt v}$, similarly the vertical minimal path $m_{\ol v} $ passes through $\ol v$.
We claim the following: for any level $n$ there exist $l_n > n$ such that for every $l \geq l_n$ and every vertex $u \in V_l$ , the maximal and minimal finite paths in $E(v_0, u)$ agree with some $M_{\wt v}$, $m_{\ol v}$ respectively on the first $n$ entries,
where the  vertices $\wt v \in \wt V$ and $\ol v \in \ol V$ depend on $u$ and $l$.
Indeed, if we assumed that the contrary holds, then we would have additional maximal (or minimal) paths not belonging to $\{M_{\wt v}: \wt v \in \wt V \}$ (or
$\{m_{\ol v}: \ol v \in \ol V \}$).
Thus, after an appropriate telescoping, we can assume that if $v$ is any vertex in $V_n,\ n\geq 2$, and  $\wt e_v$ and $\ol e_v$ are the maximal and minimal edges in the set $r^{-1}(v)$ with respect to $\om$, then $\wt e_v \neq \ol e_v$ and  $s(\wt e_v) \in \wt V_{n-1}$,  $s(\ol e_v) \in \ol V_{n-1}$. By further telescoping we can assume that the sources of $\wt e_v$ and $\ol e_v$ do not depend on the level in which $v$ lies.  If $\om$ is perfect, Remark \ref{telescoping_assumption} explains why it is possible to telescope $(B,\om)$ so that (5) is true.
\end{proof}

\begin{definition}\label{well_telescoped}
Let $B$ be a finite rank $d$ Bratteli diagram.
\begin{enumerate}
\item
If $B$  satisfies the  conditions (1) - (2) of Lemma \ref{well_telescoping_lemma}, we say that $B$ is {\em strictly rank $d$}.
\item
If $(B,\om)$ satisfies  conditions (1) - (4) of Lemma \ref{well_telescoping_lemma},  or if  $(B,\om)$ is a finite rank perfectly  ordered diagram satisfying conditions (1) - (5) of Lemma \ref{well_telescoping_lemma},we say that $(B,\om)$ is {\em well-telescoped.}
\end{enumerate}

\end{definition}

{\em For the remainder of Section \ref{language_skeleton}, we assume that unordered finite rank $d$ Bratteli diagrams  are strictly rank $d$. We assume that finite rank ordered Bratteli diagrams are well-telescoped.  }

\vspace{1em}

Thus, any ordering $\om$ determines a collection $\{M_{\wt v}, m_{\ol
  v}, \wt e_w, \ol e_w: w\in V^*\backslash V_{0},\,\, \wt v \in \wt
 V \mbox{ and } \ol v \in \ol V \}$. This collection of paths and edges
contains all information about the extremal edges of $\om$, though only partial information about $\om$ itself.  We  now extend this notion to an  unordered diagram $B$.

Let  $B$ be a strictly rank $d$ Bratteli diagram.
  We denote by $V$ the set of vertices of $B$ at each level $n\geq 1$, but
if we need to point out that this set is considered at level $n$,
then we write $V_n$ instead of $V$. For  some $k \leq d$, take two subsets
$\wt V$ and $\ol V$ of $V$ such that $|\wt V| = |\ol V| = k$.
Given any $\wt v \in \wt V$, $\ol v \in \ol V$ choose
$M_{\wt v} = (M_{\wt v}(1),..., M_{\wt v}(n),...)$ and $m_{\ol v} =
(m_{\ol v}(1),..., m_{\ol v}(n),É)$, two vertical paths in $B$
going downwards through the vertices $\wt v \in \wt V$ and $\ol v \in
\ol V$. If $v \in \ol V \cap \wt V$, then the paths $M_{v}$ and $m_{v} $
are taken such that they do not share common edges.
  Next, for each vertex $w\in V_n,  n \geq 2$, we choose two vertices $\wt v $ and $\ol v$ in $\wt V$ and $\ol V$ respectively, and for each $n\geq 2 $ and each $w \in V_n$,  distinct edges
$\wt e_w$ and $\ol e_w$ with range $w$ such that $s(\wt e_w) = \wt v$ and
$s(\ol e_w) = \ol v$ . If $w \in \wt V$ or $w \in \ol V$, then the edges $\wt e_w$ and $\ol e_w$ in $E_n$  are chosen such that $\wt e_w= M_w(n) $ and  $\ol
e_w= m_w(n)$, respectively.  We introduce the concept of a \textit{skeleton}  to create a framework for defining a perfect ordering with precisely this extremal edge structure.

\begin{definition}\label{skeleton}
 Given a  strict rank $d$ diagram $B$ and two subsets $\wt V, \ol V$ of $V$ of the
 same cardinality $k\leq d$,  a  \textit{skeleton} $\mathcal F = \mathcal F(B)$
 of $B$ is a collection $\{M_{\wt v}, m_{\ol v}, \wt e_w, \ol e_w: w \in V^*\backslash (V_0 \cup V_1 ),\,\, \wt v \in \wt V\mbox{  and } \ol v \in \ol V  \}$
 of paths and edges with the properties described above.
The vertices from $\wt V$ will be called \textit{maximal} and those from $\ol V$ \textit{minimal}.
\end{definition}

In other words, while not an ordering, a skeleton is a constrained choice of all extremal edges. As an example,
 when $\wt V = \ol V = V$, the skeleton is simply the set $\{M_{\wt
   v}, m_{\ol v}: \wt v, \ol v \, \in V\}$.
 As discussed in Lemma \ref{well_telescoping_lemma}, any well telescoped ordered finite rank  Bratteli diagram $(B, \om)$ has a natural skeleton $\mathcal F_\om$  (recall that the extremal paths are vertical).
  Conversely, it is obvious that there are several skeletons that one can define on $B$, and  for any skeleton $\mathcal F$ of a Bratteli diagram $B$ there is at least one ordering $\om$ on $B$ such that $\mathcal F = \mathcal F_\om$. A skeleton $\mathcal F_{\om}$ contains no information about whether $\om\in
\mathcal P_{B}$. Note that a skeleton does not contain information about which are the maximal edges in $E_1$; this will not impact our work.

\medskip
 Next we define a directed graph $\mathcal H = (T,
P)$ associated to a Bratteli diagram $B$ of strict finite rank and having
skeleton $\mathcal F$.  Implicit in the definition of this
directed graph is the assumption that we are working towards
 constructing perfect orderings $\om$ whose  skeleton
 $ \mathcal F_{\om}=\mathcal F $.
 Thus  we suppose that we
also have a bijection  $\sigma : \wt V \to \ol V$ that,  in the case when
$\mathcal F = \mathcal F_{\om}$ with $\om \in \mathcal P_{B}$,
 will be the bijection described in Proposition \ref{ExistenceVershikMap}, so that
 $\varphi_\om(M_{\wt v}) = m_{\sigma(\wt v)}$.

\begin{definition}
For any vertices $\wt v \in
\wt V$ and $\ol v \in \ol V$, we set
\begin{equation}\label{W} W_{\wt v} = \{w \in V : s(\wt e_w) =
\wt v\}, \ \ W'_{\ol v} = \{w \in V : s(\ol e_w) =
\ol v\}.
\end{equation}
 Then $W = \{W_{\wt v} : \wt v \in \wt V\}$ and $W' = \{W_{\wt v}' : \ol v \in \ol V\}$ are both partitions of $V$.  We call $W$ and $W'$ the {\em partitions generated by $\mathcal F$.}
 \end{definition}
 Let $[\ol v, \wt v]:= W'_{\ol v} \cap W_{\wt v}$, and
define the partition \[W \cap W' := \{[\ol v,\wt v] : \ol v\in \ol V, \wt v \in \wt V\} .\]

\begin{definition}\label{associated_graph}
Let $B$ be a strict finite rank diagram,  \[\mathcal F = \{M_{\wt v},
m_{\ol v}, \wt e_w, \ol e_w: w\in V^*\backslash (V_{0}\cup V_1),\,\, \wt v
\in \wt V\mbox{ and } \ol v \in \ol V\}\] be a skeleton on $B$,
and suppose
 $\sigma :\wt V \rightarrow \ol V$ is a bijection.
   Let the graph $\mathcal H = \mathcal H(T,P)$, have vertex set
 $$ T =
\{[\ol v, \wt v] \in \ol V\times \wt V : [\ol v, \wt v]\neq
\emptyset\},  $$  and edge set $P$, where there is an edge from  $[\ol v, \wt v]$ to $[\ol v_1, \wt v_1]$
 if and only if  $\sigma(\wt v) = \ol v_1$.
The
directed graph  $\mathcal H$  is called the  \textit{graph associated to  $(B, \mathcal F, \sigma)$}.
\end{definition}

 Note that for a fixed skeleton, different bijections $\sigma$ will define different graphs $\mathcal H$.

\begin{remark} Suppose $(B, \om)$ is a perfectly
 ordered,  well telescoped finite rank Bratteli diagram, $\mathcal F_\om$ is the
 skeleton on $B$ defined by $\om$ and $\sigma$ is  the bijection given by Proposition \ref{ExistenceVershikMap}.
Let $\mathcal H= (T,P)$ be the graph associated to $(B, \mathcal F, \sigma)$.  Let $w =
 v_1\cdots v_p$ be a word in the language $\mathcal L_{B,\om}$ and   suppose
 $v_i \in t_i$ where $t_i \in T$. Then there exists a path in $\mathcal
 H$ starting at $t_1$ and ending at $t_p$.  Moreover, the following is also true; the proof is
 straightforward and is omitted.
\end{remark}

\begin{lemma}\label{V-H-correspondence}
Let $B$ be an aperiodic, strict  finite rank Bratteli diagram, let
 $\mathcal F$ be  a skeleton on $B$, $\sigma:\wt V\rightarrow \ol V$
 be a bijection, and  let $\mathcal H=(T,P)$ be the
 associated  graph to $(B, \mathcal F,\sigma)$.
Suppose there exists an ordering $\om$
 on $B$ with skeleton $\mathcal F$,  and there is an $M$ such that
 whenever  $N>n\geq M$,   if a  word $w=v_{1}\ldots
 v_{p} \subset w(v,n,N)$ for  $v \in V_N$, then  $w$ corresponds to a path in $\mathcal H$
 going through vertices $t_{1}, \ldots t_{p}$, where  $v_{i}\in V_{n}$ belong to $t_{i}\in T$.
Then $\om$ is perfect and
 $\varphi_{\om} (M_{\wt v}) = m_{\sigma(\wt v)}$ for each $\wt v \in \wt V$.
\end{lemma}

\begin{definition}\label{aperiodic_definition}
We define the family $\mathcal A$ of Bratteli diagrams, all of whose incidence matrices are of the form
\[F_{n}:=
\left(
\begin{array}{ccccc}
A_{n}^{(1)} & 0 & \ldots  &  0 & 0 \\
0 & A_{n}^{(2)} & \ldots  & 0 & 0 \\
\vdots & \vdots  & \ddots & \vdots & \vdots     \\
0 & 0 & \ldots  & A_{n}^{(k)}  & 0 \\
B_{n}^{(1)}& B_{n}^{(2)}  & \ldots  & B_{n}^{(k)} & C_{n}  \\
 \end{array}
\right), \ \ n\geq 1,
\]
where
\begin{enumerate}
\item
for $1\leq i \leq k$ there is some $d_{i}$ such that for each $n\geq 1$,
$A_{n}^{(i)}$ is a $d_{i}\times d_{i}$ matrix,
\item all  matrices $A_{n}^{(i)}$, $B_{n}^{(i)}$ and $C_{n}$ are strictly
 positive,
\item $C_{n}$ is a $d\times d$ matrix,
\item there exists $j\in \{\sum_{i=1}^{k}d_{i} +1, \ldots \sum_{i=1}^{k}d_{i} +d\}$ such that for each $n\geq 1$, the $j$-th row of $F_{n}$ is strictly positive.
\end{enumerate}
If a Bratteli diagram's incidence matrices are of the form above, we shall say that it has $k$ {\em minimal components}.
\end{definition}

As shown in \cite{bezuglyi_kwiatkowski_medynets_solomyak:2011}, the family  $\mathcal A$ of diagrams corresponds to aperiodic homeomorphisms of a Cantor set that have exactly $k$ minimal components with respect to the tail equivalence relation $\mathcal E$.

Recall that a directed graph is {\em strongly connected} if for any two
vertices $v$, $v'$, there is a  path from $v$ to $v'$, and also  a path from $v'$
to $v$. If at least one of these paths exist, then $G$ is {\em weakly
connected}, or just {\em connected}. We notice that, given $(B, \mathcal F, \sigma)$,  an associated graph $\mathcal H =(T,P)$ is not connected, in general.

\begin{proposition}\label{connectedness and goodness}
 Let $(B,\om)$ be a finite rank,  perfectly ordered and well telescoped  Bratteli diagram, and suppose $\om$ has skeleton $\mathcal F_\om$ and permutation $\sigma$.
\begin{enumerate}
\item If $B$ is simple, then the associated graph $\mathcal H$ is
strongly connected.
\item If $B \in \mathcal A$, then the associated graph $\mathcal H$ is
weakly connected.
\end{enumerate}
\end{proposition}

\begin{proof} We prove (1); the proof of (2) is similar, if we focus on $w(v,n-1,n)$ where $v$ is the vertex which indexes the strictly positive row in $F_{n}$. Recall that in addition to assuming that $(B,\om)$ is well telescoped, since $\om$ is perfect, we assume we have telescoped so that all entries of $F_n$ are positive for each $n$, and also so that if   $\wt v \ol v $ is a subword of $w(v,m,n)$ for $1\leq m<n$, then $\sigma(\wt v) = \ol v$.
We need to show that for any two vertices $t = [\ol v, \wt v]$ and $t'= [\ol v', \wt v']$ from the vertex set $T$ of $\mathcal H$, there exists a path from $t$ to $t'$.
\medskip

\textit{Claim 1}.  Let $n>2$ and $w(u, n-1, n) = v_1\cdots v_k$ be a word where $v_i \in [\ol v_i, \wt v_i], i=1,..., k$. Then there is a path from $[\ol v_1, \wt v_1]$ to $[\ol v_k, \wt v_k]$ going through the vertices $[\ol v_i, \wt v_i],\ i=1,..., k$, in that order.
\medskip

For, given $1\leq i \leq k-1$, since   $v_i v_{i+1}$ is a subword of $w(u,n-1,n),$ then the concatenation of the two words
$w(v_i,n-2,n-1)w(v_{i+1},n-2,n-1)$ is a subword of $w(u,n-2,n)$, so that $\wt v_i \ol v_{i+1}$ is a subword of $w(u,n-2,n)$. By our telescoping assumptions,  $\sigma(\wt v_i) = \ol v_{i+1}.$

Now, let $T^*$ be the  subset of $T$ of vertices of the form $[\ol v, s(\wt e_{\ol v})]$ where $\ol v \in \ol V$. (Note that  $ [\ol v, s(\wt e_{\ol v})]\neq \emptyset$ since $\ol v \in [\ol v, s(\wt e_{\ol v})]$.)
 It is obvious that there is an edge from $t= [\ol v, \wt v]$ to $t^*= [\sigma(\wt v), s(\wt e_{\sigma(\wt v)})]$ in $\mathcal H$. \medskip

\textit{Claim 2}. For any  $t^* \in T^*$ and $t'= [\ol v', \wt v'] \in T$, there is a path from $t^*$ to $t'$.

\medskip

 To see that this, we
will use Claim 1. Let $t^* =  [\ol v^*, \wt v^*]$ where $\wt v^*  =s (\wt e_{\ol v^*}).$   Let $v\in V_{n-1}$ belong to $t' $ in $\mathcal H$. By the simplicity of $B$, there exists an edge $e \in
E(v, \ol v^*)$ where $\ol v^* \in V_n$. Thus $w(\ol v^*, n-1, n) = \ol v^* \ldots v \ldots \wt v^*$. If $n>2$, then by
Claim 1 that there is a path from $t^*$ to $t'$.

To complete the proof of the proposition, we concatenate the paths from $t$ to $t^*$ and from $t^*$ to $t'$ in $\mathcal H$.
\end{proof}

\begin{remark} It is not hard to see that the converse statement to Proposition \ref{connectedness and goodness} is not true. There are examples of non-simple perfectly ordered diagrams of finite rank whose associated graph is strongly connected.

Note also that the
assumption that $\om$ is perfect is crucial.
Moreover, there are examples of \textit{simple} finite rank Bratteli
 diagrams and skeletons whose associated graph is not   strongly
connected. Indeed,
 let $B$ be a simple stationary  diagram with $V = \{a,b,c\}$ with the skeleton
 $\mathcal F = \{M_a, M_b, m_a, m_b; \wt e_c, \ol e_c\}$ where $s(\wt
 e_c) =b, s(\ol e_c) = a$. Let $\sigma(a) = a, \sigma(b) =
 b$. Constructing the associated graph $\mathcal H$, we see that there
 is no path  from $[b,b]$ to $[a,a]$. It can be also shown that there is
 no perfect ordering $\om$ such that $\mathcal F = \mathcal F_\om$. This
 observation complements  Proposition \ref{connectedness and goodness}
 by stressing the importance of the strong connectedness of $\mathcal H$ for the existence of perfect orderings.
\end{remark}

We illustrate the definitions of skeletons and associated graphs with several examples that will be also used later.

\begin{example}
 \label{examples of graphs}
 Let $(B, \om)$ be an ordered Bratteli diagram of strict rank $d$, where $V = \{1, \ldots ,d\}$, and where $\om$ has $d$ vertical maximal and $d$ vertical minimal paths. Then the skeleton $\mathcal F_\om$ is formed by pairs of vertical paths $(M_i, m_i)$ going downward through the vertex $i \in \{1,..., d\}$.

Let $\sigma$ be a permutation of the set $\{1,2, \ldots , d\}$. The graph $\mathcal H$  is represented as a disjoint union of connected subgraphs generated by cycles of $\sigma$.
 If $\om$ is
perfect, then by Proposition \ref{connectedness and goodness},
$\sigma$ is cyclic.  In this case, $[i,i] =\{i\}$, so vertices
of $\mathcal H$ are $\{[i,i]: 1\leq i\leq d\}$, and there is an edge
from $[i,i]$ to $[j,j]$ if and only if $j = \sigma(i)$. Thus, the
structure of $\mathcal H$ is represented by the cyclic permutation
$\sigma$.

\end{example}

\begin{example}
Let $\mathcal F$ be a skeleton on a simple  strict rank $d$  diagram $B$
such that $V=\{1, \ldots,d-1, d\}$  and $\wt V = \ol V = \{1,\ldots, d-1\}$. Depending on $\sigma$, the  associated graph $\mathcal H$ that can be  associated to $\mathcal F$ is one of two kinds.

\begin{enumerate}\item Suppose $s(\wt e_d) = s(\ol e_d) = j$ where $1\leq j\leq d-1$. Then $[i,i] = \{i\}$ for $1\leq i
\leq d-1$, $i\neq j$, and $[j,j] = \{j,d\}$. In $\mathcal H$ then the
vertex set is  $T=\{[i,i]: 1\leq i \leq d-1\}.$ For $\mathcal H$ to be
strongly connected, $\sigma$ must be a cyclic permutation of $\{1,\ldots ,
d-1\}$, and in this case there is an edge from $[i,i]$ to $[i',i']$ if
and only if $i'= \sigma(i)$.

\item Suppose $s(\wt e_d) = j\neq s(\ol e_d) = i$ where $1\leq i, j \leq d-1$; we can assume that $i < j$. Here  $[l,l] = \{l\}$ for $1\leq l \leq d-1$ and $[i,j] = \{d\}$, so that
$T = \{[l,l]: 1\leq l \leq d-1\}\cup \{[i,j]\}$. Here also, for $\mathcal H$ to be
strongly connected, $\sigma$ must be a cyclic permutation of $\{1, \ldots ,
d-1\}$, and the edges described in (1) form a subset of $P$.  In addition there is  an edge from  $[\sigma^{-1}(i), \sigma^{-1}(i)]$ to $[i,j]$, and also an edge from $[i,j]$ to $[\sigma (j),\sigma(j)]$. If $\sigma(j)=i$, then there is also a loop at $[i,j]$.
\end{enumerate}

\end{example}
\begin{example}
 \label{stationary_example} We continue with Example \ref{stationary_example_a}.
 Since  $\varphi_\om(M_a) = m_a$, $\varphi_\om(M_b) =
m_b$, this means that  $\sigma(a) =a,
\sigma(b) =b$. Noting that $s(\wt e_c) = b, s(\wt e_d) = a, s(\ol e_c)
= a, s(\ol e_d) = b$, we have the completely determined skeleton
$\mathcal F_\om$.
Note that the
vertices $T$ of $\mathcal H$ are $[a,a] = \{a\}, [a,b]
= \{c\}, [b,a] = \{d\}$ and $[b,b] = \{b\}$. The associated
graph $\mathcal H$ is shown in Figure
\ref{stationary_associated_graph}.
\end{example}

\begin{example}
Let $V  = \{v_{1},v_{2},v_{1}^{*},v_{2}^{*}, w_{1},w_{2}\}$ and
     $\wt V = \ol V =\{v_{1}, v_{1}^{*}\} $; let $\sigma(v_{1})=v_{1}$ and $\sigma(v^*_{1})=v^*_{1}$.
Suppose that $W_{v_{1}}'=\{v_{1},v_{2},w_{1}\}$,
     $W_{v_{1}}=\{v_{1},v_{2},w_{2}\}$,
$W_{v_{1}^{*}}'=\{v_{1}^{*},v_{2}^{*},w_{2}\}$  and
     $W_{v_{1}^{*}}=\{v_{1}^{*},v_{2}^{*},w_{1}\}$.
Then the associated graph $\mathcal H$ is strongly connected. We remark
     that this can be the skeleton of an aperiodic diagram with two
     minimal components living through the vertices $\{v_{1},v_{2}\}$
     and $\{v_{1}^{*}, v_{2}^{*}\}$ respectively.
\end{example}

\begin{figure}[h]
\centerline{\includegraphics[scale=1.1]{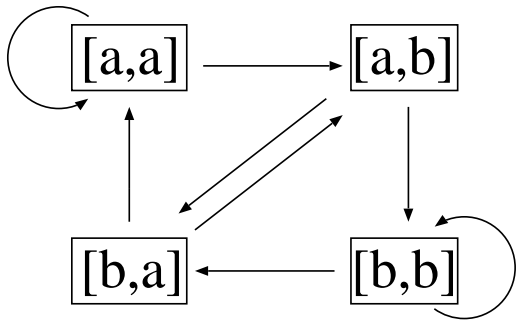}}
\caption{The graph associated to $\mathcal F_\omega$ in Example \ref{stationary_example}
  \label{stationary_associated_graph}}
\end{figure}

We illustrate the utility of the notions of skeleton and accompanying directed graphs in the following results, which give sufficient conditions for an ordering $\om$ to  belong to
$\mathcal P_B^{\,c}$. Even though these are conditions on $\om$, some diagrams $B$ force this condition on all  orderings in ${\mathcal O}_{B}$ - this is the content of Proposition
\ref{corollary_1}.

\begin{proposition}\label{curious}
Let  $(B, \om)$ be a perfectly ordered, well telescoped rank d Bratteli diagram.
Suppose that $\omega$ has $k$ maximal and $k$ minimal paths, where $1<k\leq d$. Then
for   some $v\in V$,
$vv\,\not\in \mathcal L_{B, \om}$.

\end{proposition}

\begin{proof}
Let $\om$ have skeleton
 $\mathcal F_\om = \{M_{\wt v},
m_{\ol v}, \wt e_w, \ol e_w: w\in V^* \backslash (V_{0}\cup V_1),\,\, \wt v
\in \wt V\mbox{ and } \ol v \in \ol V\}$, and
suppose that $\om$  is perfect. Then there exists a bijection
$\sigma$ of $\{1,\ldots ,k\}$ such that $\sigma (i) = j$ if and only
if $v_{M_i}v_{m_j}\,\in \mathcal L_{B,\om}$.
Suppose that for each $v$, there is some some $v^*$ such that   $vv\,\in w(v^*,n,n+1)$ for infinitely many $n$.
We claim that $V=
\bigcup_{i=1}^{k} [v_{m_{\sigma(i)}}, v_{M_i}]$.
For if $s(\ol e_{v}) = v_{m_j}$ and $s(\wt e_{v}) =
v_{M_{i}}$, then if $vv\in w(v^{*},n,n+1)$, this implies that
$v_{M_i}v_{m_{j}}\in w(v^{*},n-1,n+1)$. Since this occurs for infinitely many $n$, then  Proposition
\ref{ExistenceVershikMap} tells us that $j=\sigma(i)$.

Since $W$ and $W'$ are both partitions of $V$,
 the relation $V= \bigcup_{i=1}^{k} [v_{m_{\sigma(i)}}, v_{M_i}]$ actually  means that $W_{v_{M_i}} =  W_{v_{m_{\sigma(i)}}}^{'}$ for every $i$.
It follows  that the associated graph $\mathcal H$ has the following simple form: the vertices of $\mathcal H$ are $[v_{m_{\sigma(i)}}, v_{M_i}], i=1,...,k,$ and the edges are given by $k$ loops, one  around each vertex. Since $k>1$, this means $\mathcal H$ is not connected, contradicting Proposition \ref{connectedness and goodness}.

\end{proof}

%Proposition \ref{curious} is a special case  of the following more general result.

%\begin{proposition}\label{curious2}
%Let  $B$ be a rank $d$ Bratteli diagram.
%Suppose that $\om\in {\mathcal P}_{B}\cap \mathcal O_B(k)$, $1<k\leq d$ 
%is such that for the corresponding
%skeleton $\mathcal F_\om$ we have
%\begin{enumerate}

%\item$W= W'$, and

%\item  there is a partition of  the set $\{1,\ldots k\}$ into at least two subsets, such that if
%$\{i(1),\ldots i(l)\}\subset \{1,\ldots k\}$ is one such subset, then
%$W_{v_{M_{i(j)}}} = W'_{v_{m_{\sigma(i(j-1))}}}$ for $ 2\leq j\leq \ell$, and
%$W_{v_{M_{i(1)}}} = W'_{v_{m_{\sigma(i(\ell))}}}$.

%\end{enumerate}
%Then $B$ is a disjoint union of some of its subdiagrams.
%\end{proposition}

%The proof of this result is based on the same idea as that of Proposition
% \ref{curious}. Indeed, in Proposition \ref{curious}, we have the equalities 
%$W_{v_{M_{j}}}=W'_{v_{m_{\sigma(j)}}}$ and therefore deal with the partition of
 %$\{1,\ldots k\}$ given by $\{ \{1\}, \ldots , \{k\}\}$.  We note only that in 
%the more general situation the associated graph $\mathcal H$ is now a 
%disjoint union of circles: each set $\{i_{1},\ldots ,i_{l}\}$ defines a subgraph 
%as mentioned above whose vertices are linked by a cyclic permutation 
%defined by $\sigma$.
%\medskip

\subsection{Bratteli diagrams that support no perfect orders}\label{no_perfection}

The next proposition describes how for some aperiodic diagrams $B$ that
belong to the special class $\mathcal A$ (see Definition \ref{aperiodic_definition}), there are structural obstacles to the existence of perfect orders on $B$. This
is a generalization of an example in \cite{medynets:2006}.

\begin{proposition}\label{corollary_1}
Let $B\in \mathcal A$ have $k$ minimal components, and such that for each $n\geq 1$,
 $C_{n}$ is an $s\times s$ matrix where $1\leq s\leq k-1$.
If $k=2$, there are perfect orderings on $B$ only if $C_{n}=(1)$ for
all but finitely many $n$. If $k>2$, then there is no perfect ordering on $B$.

\end{proposition}

\begin{proof} We use the notation of Definition \ref{aperiodic_definition} in this proof.  Let $V^{i}$ be the subset of vertices corresponding to the subdiagram
defined by the matrices  $A_{n}^{(i)}$ for $i=1, \ldots k$, and let  $V^{k+1}$ be the subset of vertices corresponding to the subdiagram
defined by the matrices $C_{n}$. Suppose that $\om$ is a perfect ordering on $B$, and we assume that    $(B, \om)$ is well telescoped and  has skeleton $\mathcal F_\om$. (Otherwise we work with the diagram $B'$ on which $L(\om)$ is well telescoped: Note that if $B$ has
incidence matrices of the given form, then so does any
telescoping.)
Note that $|\ol V| = |\wt V| \geq k$ since each minimal
component has at least one maximal and one minimal path.  Also, if $\wt v \in V^{i}$, then $\sigma(\wt v ) \in V^{i}$.
 There are
$k$ connected components of vertices $T_{1},
\ldots T_{k}$, such that there are no edges from vertices in $T_{i}$
to vertices in $T_{j}$ if $i\neq j$. To see this, if $1\leq i \leq k$, let
$T_{i} = \{[\ol v, \wt v]: \ol v \in V^{i}, \wt v\in
V^{i}\}$.

 If $k=2$,
there are no extremal paths going through $c$, the unique vertex in
$V^{3}$ - otherwise there would be
 disjoint components in $\mathcal H$, and since $\om$ is perfect,
  this would  contradict  Proposition \ref{connectedness and goodness}.
  So $c \in [\ol v, \wt v]$ where $\ol v \in V^{i}$ and  $\wt v\,\in V^{j}$ for some  $i\neq j$. Thus in $H$ there are paths  from vertices in $T_{i}$ to vertices in $T_{j}$ through $c$, but not back again. The only way this can occur validly is if $C_{n}=(1)$ for all  large $n$.

 If $k>2$, then there are at most  $k-1$ vertices remaining in $\mathcal
H$, outside of the components $T_1, \ldots, T_k$. We shall argue that even in the extreme case, where there are $k-1$ such vertices, there would  not be  sufficient connectivity in $\mathcal H$ to support an $\om \in \mathcal P_B$.  Call these $k-1$ vertices $t_1, \ldots t_{k-1}$, where $t_i=[\ol v_i, \wt v_i]. $ If  $V^{k+1}=\{ v_1, \ldots , v_{k-1}\}$, we have labeled so that  $v_i \in t_i$. For each one of these  vertices $t_i$ there are
incoming edges from vertices in at most one of the components $T_{j}$, for
$1\leq j \leq k$, and also outgoing edges to vertices in  at most one of the
 components $T_{j'}$, for $1\leq j' \leq k$. So at least one of the
 components, say $T_{1}$, has no incoming edges with source outside
 $T_{1}$.

 Suppose first that each $t_i =[\ol v_i, \wt v_i]$ satisfy $\ol v_i \in V^1$, in which case all other $T_i's$ have no outgoing edges. But then for $T_i\neq T_j$, $i\neq j$, $i,j\neq 1$, there is neither a path from $T_i$ to $T_j$, nor from $T_j$ to $T_i$.  This contradicts the second part of  Proposition \ref{connectedness and goodness}.

 Suppose next that for some $i$, $t_i=[\ol v_i, \wt v_i]$ and $\ol v_i \not\in V^1$. Since $\wt v_i \not \in V^1$, there is no edge between $t_i$ and $V_1$. Since $B_n^{(1)}$ has strictly positive entries, $w(v_i, n, n+1)$ must contain occurrences from vertices in $V^1$; and these occurrences have to occur somewhere in the interior of the word. But this contradicts the fact that $T_1$ has no incoming edges from outside $T_1$.

\end{proof}

In the above proposition, the extreme case - when there are $k$
extremal pairs, and the vertex set of $\mathcal H$ has size $ 2k-1$ - still does not produce
perfect orderings, but only just, as the next proposition
demonstrates. First we define the family $\mathcal M$ of matrices whose relevance will become apparent in Theorem \ref{d max/min paths}.

\begin{definition} \label{M-matrix}
Let $\mathcal M$ be the family of  matrices whose entries take values in $\N$, and which are  of  the form
\begin{equation}\label{D-matrix}
F= \left(
      \begin{array}{cccc}
        f_1 +1 & f_1 & \cdots & f_1 \\
        f_2     &  f_2 +1 &  \cdots & f_2 \\
        \vdots & \vdots & \ddots & \vdots \\
        f_d & f_d & \cdots &  f_d +1 \\
      \end{array}
    \right)
\end{equation}
for some $d \in \mathbb N$.
\end{definition}

\begin{proposition}\label{corollary_2}
Let $B \in \mathcal A$ be a   Bratteli diagram with $k$ minimal subcomponents, and where for each $n\geq 1$, $C_{n}$  is a $k\times k$ matrix.
If $(B,\om)$ is a perfectly ordered, well telescoped  Bratteli diagram  with skeleton $\mathcal F_\om$, then $C_{n}\,\in \mathcal M$ for all $n$.
\end{proposition}

\begin{proof}
We use the notation of Proposition \ref{corollary_1}. The proof of this last proposition showed us that for a perfect order to be supported by $B$, each component $T_i$ has to have an incoming edge from outside $T_i$. Similarly, each component $T_i$ has to have an outgoing edge with range outside of $T_i$.
Label $V^{k+1}=\{v_1, \ldots v_k \}$ so that $v_i \in [\ol v_i, \wt v_{h(i)}]$ where $\ol v_i \in T_i$ and $h:\{1, \ldots , k\} \rightarrow \{1, \ldots , k\}$ is a bijection.
Thus in $\mathcal H$, from each $T_i$ there is an edge from $T_i$ to $[\ol v_i, \wt v_{h(i)}]$, and  there is an edge from $[\ol v_{h^{-2}(i)}, \wt v_{h^{-1}(i)}]$ to $T_i$. In addition for each $i$, there is  (possibly) an edge from $[\ol v_i, \wt v_{h(i)}]$ to $[\ol v_{h(i)}, \wt v_{h^2(i)}]$. See Figure \ref{last_pic} for an example of such a graph.

\begin{figure}[h]
\centerline{\includegraphics[scale=0.9]{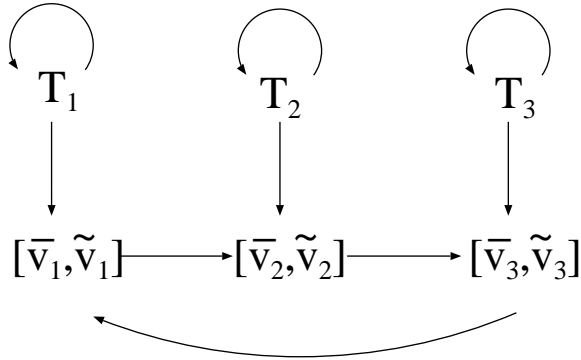}}
\caption{An example of $\mathcal H$ when $B$ has 3 minimal subcomponents and $h =(123)$.
  \label{last_pic}}
\end{figure}

 If $h$ is not a cyclic permutation, then the graph $\mathcal H$ is disconnected, in which case there are no perfect orders on $B$ which have  the skeleton $\mathcal F_\om$. Thus $h$ must be cyclic, and inspection of the graph $\mathcal H$ tells us that for each $v_i \in V^{k+1}$, and for each $n$,
 $v_{i}\,\in [{\ol v}_{i}, {\wt v}_{h(i)}]$ and
$$
w(v_{i},n-1,n) = \left(\prod_{j=1}^{k_n}W_i^{(j)} v_{i} W_{h(i)}^{(j)} v_{h(i)} \ldots  W_{h^{-1}(i)}^{(j)} v_{h^{-1}(i)}  \right)   W_i v_i W_{h(i)}
$$
where $\prod$ refers to concatenation of words, each $W_i^{(j)}$ is a (possibly empty) word with letters in $V^{i}$,  and $W_i$, $W_{h(i)}$ are non-empty words. The result follows.
\end{proof}

\subsection{Perfect orderings that generate odometers}\label{odometer}

\begin{definition}\label{odometer_definition}
If a minimal Cantor dynamical system $(Y, T)$ admits an adic  representation by a Bratteli diagram $B $ with $|V_n| =1$ for all levels $n$, then $T$ is called an \textit{odometer}.
\end{definition}

Let $\mathcal L \subset \mathcal A^{\mathbb N}$. A word $W\in \mathcal L
$ is {\em periodic } if it can be written as a concatenation $W=U^{k}$
of $k$ copies of a word $U$
where $k>1$. Given a word $W=w_{1}\ldots w_{p}$, we define
$\sigma^{i}(W):= w_{i+1}w_{i+2}\ldots w_{p}w_{1}\ldots w_{i}$.
We say that $\mathcal L$ is
{\em periodic} if there is some word $V\in \mathcal L$ such that any
word $W\in \mathcal L$ is of the form $SV^{k}P$ for some suffix (prefix)
 $S=S(W)$ ($P=P(W)$) of $V$.  Finally if $\mathcal Q=\{q_{1}, q_{2}, \ldots q_{n}\}$ is a partition of a set $X$ and
 $T: X\rightarrow X$ is a bijection, then we say that $\mathcal Q$ is {\em periodic}
for $T$ if $T(q_{i})= q_{i+1}$ for $1\leq i <n$ and $T(q_{n})=q_{1}$.

Next we state and prove a result which Fabien Durand has communicated
to us as a known result; the proof below is a direct generalisation of
 the  proof of Part (ii) of Proposition 16 in
\cite{durand_host_scau:1999}.

\begin{proposition} \label{adic_odometers}
Let $\om$ be a perfect ordering on the simple, strict  finite rank diagram $B$. If ${\mathcal L}_{B,\om}$ is periodic, then
$(X_{B}, \varphi_{\om})$ is topologically conjugate to an odometer.
\end{proposition}

\begin{proof}
Suppose ${\mathcal L}_{B,\om}$ is periodic. Let $V$ denote the vertex set of $B$ at each level.  Fix $\ol v $ such that there is a vertical minimal path going through the vertex $\ol v$.
Then for all $k$, $\lim_{n\rightarrow \infty} w(\ol v , k, n) $ exists. In particular
$\lim_{n\rightarrow \infty} w(\ol v , 1, n) = WWW\ldots $ where
$W = w_{1}w_{2} \ldots w_{p}$ is of length $p$ and is not periodic.

We define a sequence of partitions $({\mathcal Q}_n)$ that will be  refining,
clopen, generating periodic partitions of $(X_{B}, \varphi_{\om})$, and
 such that $|\mathcal Q_{n+1}|$ is a multiple of $|\mathcal Q_{n}|.$ The
 existence of this sequence implies that $(X_{B}, \varphi_{\om})$ is an
 odometer. For $x = x_{1}x_{2} \ldots \in X_{B}$ (where
 $s(x_{1})=v_{0}$), $j\in \mathbb N$, and $0\leq i\leq p-1$,  let
\[[i]_{j}= \{x: s(x_{j+1}) \, s((\varphi_{\om} (x))_{j+1})\,\ldots
s((\varphi_{\om}^{p-1}(x))_{j+1} = \sigma^{i}(W)\}.\]
Let
\[\mathcal Q_{1}:= \{[i]_{1}: 0\leq i\leq  p-1\} .\]
Since $W$ is not periodic each $x$ lives in only one $[i]_{1}$,
 and
 $\mathcal Q_{1} $
is of period $p$ for $\varphi_{\om}$.

Given a vertex $v\in V_{n}$, recall that
$ h_{v}^{(n)}= |E(v_{0}, v)|$ for $v\in V_{n}$.
 Define for $n>1$
\[\mathcal Q_n:=\{[i_{1},i_{2}]: 0\leq i_{2}\leq p-1, \,\,\,
0\leq i_{1}\leq h_{w_{i_{2}+1}}^{(n)}-1 \} \]
 where
\[[i_{1},i_{2}]:=
 [i_{2}]_{n} \cap \{x: x_{1}x_{2}\ldots x_{n}\in E(v_{0}, w_{i_{2}+1})
\mbox{ and has  $\om$-label $i_{1}$   }\}.
\]
Then for each $n\geq 1$, $\mathcal Q_{n}$ is a clopen partition,   $\mathcal
 Q_{n+1}$ refines $\mathcal Q_{n}$, and it is clear that $(\mathcal
 Q_{n})$ is a generating sequence of partitions. We claim that
 $\mathcal Q_{n}$ is $\varphi_{\om}$ periodic. For, if
$i_1 < h_{w_{i_{2}+1}}^{(n)}-1$,
$\varphi_{\om}([i_1 ,
 i_2]) =([i_{1}+1,i_{2}])$.
If $i_1 = h_{w_{i_{2}+1}}^{(n)}-1$ and $i_{2}<p-1$ then $\varphi_{\om}([i_1 ,
 i_2])= [(0, i_{2}+1)]$. Finally
$\varphi_{\om}([h_{w_{i_{2}+1}}^{(n)}-1, p-1]) =[0,0].$

It remains to show that $|\mathcal Q_{n+1}|$ is a multiple of $|\mathcal
 Q_{n}|$.
 Note that  for each $v\, \in V$ and each $n\geq 2$, $w(v,n-1,n) = S_{v}^{(n)}
 W^{\alpha_{v}^{(n)}}P_{v}^{(n)}$ with $S_{v}^{(n)}$ a proper suffix of $W$, $P_{v}^{(n)}$ a proper prefix of $W$, and, whenever $vu\in \mathcal L(B,\om)$, then
$P_{v}^{(n)}S_{u}^{(n)}$ is either empty or equal to $W$. Note that   $w_{p}\ol w_{1} \in \mathcal L(B,\om)$, so that for each $n,$ $P_{w_p}^{(n)} S_{w_{1}}^{(n)} = W$ or is the empty word. We assume that $P_{w_p}^{(n)} S_{w_{1}}^{(n)} = W$ in the computation below, otherwise simply  remove the `1'.
If $W'\subset W$, let
$\#_{W'}(W)$ denote the distinct number of occurrences of $W'$ in W.
 Then
\begin{eqnarray*}
|\mathcal Q_{n+1}| & = & p \sum_{v \in W} \#_v(W) h_v^{(n+1)}
\\ & = &  p \sum_{v \in W} \#_v(W) \left[    \alpha_v^{(n+1)} + \sum_{v_1w_1: P_{v_1}^{(n+1)} S_{w_1}^{(n+1)}=W }\#_{v_1w_1}(W) +1 \right] |\mathcal Q_{n}|.
\end{eqnarray*}

\end{proof}

We will now consider in detail the class of finite rank diagrams
described in Example \ref{examples of graphs}. Let the
Bratteli diagram $B$ have strink rank $d>1$.
We show that if $B$ is to support
a perfect ordering with $d$ maximal and $d$ minimal paths, then a
certain structure is imposed on the incidence matrices of $B$.

\begin{definition}
Denote by $\mathcal D$ the set of rank $d$ simple Bratteli diagrams $B$ where $V_n=\{v_1, \ldots v_d \}$ for each $n\geq 1$, and
whose incidence matrices  $(F_n)$ eventually belong to the class $\mathcal M$ (see Definition \ref{M-matrix}), and
where all entries are non-zero.
\end{definition}
It is not hard to check that the set
$\mathcal D$ is invariant under telescoping of diagrams.

\begin{proposition}\label{d max/min paths}
Let $B$ be a simple, strict rank $d$ Bratteli diagram.

\begin{enumerate}
\item
Suppose $B\in \mathcal D$, and  $\sigma$ is  a cyclic permutation  of the set $\{1, 2,...,d\}$.
Then there exists an ordering $\om\in \mathcal P_B \cap \mathcal O_B(d) $ on $B$ such that
$$
X_{\max}(\om) = \{M_1,...,M_d\}, \ \  X_{\min}(\om) = \{m_1,...,m_d\}
$$
where  $M_i$  ($m_j$) is an eventually  vertical path through the vertex
$v_{i}$  ($v_{j}$, respectively), $i, j = 1,...,d$. Moreover, the corresponding Vershik map $\varphi_\om$ satisfies the condition
\begin{equation}\label{Vershik map and permutation}
\varphi_\om(M_i) = m_{\sigma(i)}.
\end{equation}
\item  Suppose there exists an ordering $\om \in \mathcal P_B \cap \mathcal O_B(d)$ such that all maximal and minimal paths are eventually vertical. Then the Vershik map $\varphi_\om$ determines a cyclic permutation on the set $\{1,...,d\}$ and   $B$ belongs to $\mathcal D$.
\end{enumerate}
\end{proposition}

\begin{proof} (1) We need to construct a perfect ordering $\om$ on $B$ such that
(\ref{Vershik map and permutation}) holds. For every $v_{j} \in
\{v_{1},...,v_{d}\} = V_n$ and every $n $ large enough, we take $d$ subsets $E(v_{i}, v_{j})$ of
$r^{-1}(v_{j})$ where $v_{i}  \in V_{n-1}$. Then $|E(v_{i}, v_{j})| =
f_j^{(n)}$ if $i \neq j$ and $|E(v_{j}, v_{j})| = f_j^{(n)} + 1$. Hence
$|r^{-1}(v_{j})| = df_j^{(n)} +1$. for each $n\geq 1$ and each $v_{j}\,\in V_{n}$ define the order on $r^{-1}(v_{j})$ as follows:
\begin{equation}\label{periodic}
w(v_{j},n-1,n)= (v_{j}\,v_{\sigma(j)}\,v_{\sigma^{2}(j)}\ldots
v_{\sigma^{d-1}(j)})^{f_{j}^{(n)}}v_{j}\, .
\end{equation}
Clearly, relation (\ref{periodic}) defines explicitly a linear order on $r^{-1}(v_j)$.
To see that   $\varphi_\om$ is continuous,
it suffices to note that for each $j$ there is a unique $i:=\sigma(j)$ such that $v_{j}v_{i}\,\in {\mathcal L}_{B,\om}$. By Proposition \ref{ExistenceVershikMap} we are done.

(2) Conversely, suppose that $\om$ is a perfect ordering on
$B$ with $d$ maximal and $d$ minimal eventually vertical paths, so that each
vertex has to support both a maximal and a minimal path $M_{i}$ and
$m_{i}$; thus for each $i$ and each $n$ large enough, the word $\om(v_{i},n-1,n)$
starts and ends with $v_{i}$. Since $\om$ is perfect then by Proposition
\ref{ExistenceVershikMap} there is a permutation $\sigma$ such that
for each $j\,\in \{1,\ldots d\}$ only $v_{j}v_{\sigma(j)}\,\in {\mathcal L}_{B,\om}$.
So, for each $j$ and all but finitely many $n$, there is a $f_{j}^{(n)}$ such that
\begin{equation}\label{periodic1}
w(v_{j},n-1,n)= (v_{j}\,v_{\sigma(j)}\,v_{\sigma^{2}(j)}\ldots
v_{\sigma^{d-1}(j)})^{f_{j}^{(n)}}v_{j}\, .
\end{equation}

Since $B$ is simple, $\sigma$ has to be cyclic so that all vertices occur in the right hand side of (\ref{periodic1}). From (\ref{periodic1}) it also follows that all but finitely many of the incidence matrices of $B$ are of the form (\ref{D-matrix}).
\end{proof}
%I stopped here.
\begin{corollary}\label{bad_ordering}
Let $B$ be a simple Bratteli diagram of rank $d\geq 2$ and let $\om\in \mathcal P_B \cap \mathcal O_B(d)  $.
Then  $(X_{B},\varphi_{\om})$ is
conjugate to an odometer.
\end{corollary}

\begin{proof} We can assume that  $(B,\om)$ is well telescoped (conjugacy of two adic systems is invariant under telescoping of either of them).
Note that the proof of Proposition \ref{d max/min paths} tells us that
 $\mathcal L(B,\om)$ is periodic.  Lemma \ref{adic_odometers} tells us
 that $(X_{B},\varphi_{\om})$ is
conjugate to an odometer; however in this specific case there is a
 simpler sequence of periodic, refining, generating partitions
 $(Q_{n})$: let
 $ Q_{n}$ be the clopen partition
 defined by the first $n$ levels of $B$, and write
$Q_n=
 \coprod_{i=1}^{d} Q_{n}(v_{i})$, where $Q_{n}(v_{i})$ is the set of all
 paths from $v_{0}$ to $v_{i}\,\in V_{n}$.  Each non-maximal path
 in $Q_{n}(v_{i})$ is mapped by $\varphi_{\om}$ to its successor in $Q_n(v_i)$. For $i\,\in
 \{1, \ldots , d\}$, let $M_{i}^n$ denote the maximal path in $Q_{n}(v_{i})$.
 Since the ordering $\om$ is perfect,
 $\varphi_{\om}(M_{i}^n)= m_{\sigma(i)}^n$, where $m_{\sigma(i)}^n$ is the minimal path in
 $Q_{n}(v_{\sigma(i)})$. Thus the partition $Q_{n}$ is
 $\varphi_{\om}$-periodic.
 We will also compute the sequence $(k_{n})$ such
 that $|Q_{n+1}| = k_{n} |Q_{n}|$. By Proposition \ref{d max/min paths}, the incidence matrices of $B$ are of the form
(\ref{D-matrix}):  all columns of $F_{n}$ sum to the same constant $k_{n}=(1+\sum_{i=1}^{d}f_{i}^{(n)})$. Let $F_n = (f^{(n)}_{i,j})$ and
$h_{i}^{(n)}:=|Q_{n}(v_{i})|$, then $h_{i}^{(n+1)} = \sum_{j=1}^d f^{(n)}_{i,j} h_{j}^{(n)}$ and

\begin{eqnarray*}
     |Q_{n+1}| &= &\sum_{i=1}^{d}h_{i}^{(n+1)}
     = \sum_{i=1}^{d}  \left[ h_{i}^{(n)}+\sum_{j=1}^{d}h_{j}^{(n)} f_{i}^{(n)}   \right] \\
     &= & |Q_{n}| + \sum_{i=1}^{d} f_{i}^{(n)} \sum_{j=1}^{d} h_{j}^{(n)}
 = |Q_{n}|(1+\sum_{i=1}^{d}f_{i}^{(n)}).
  \end{eqnarray*}

\end{proof}

Next we consider conditions for a Bratteli diagram $B$ of strict rank $d$ to support a perfect ordering $\om$ such that $(X_B, \varphi_\om)$ is  an odometer.
Suppose  that we are given a
skeleton ${\mathcal F}$ on $B$: we have subsets $\wt{V}$ and
$\ol{V}$ of $V$,  both of cardinality $k\leq d$, a bijection $\sigma: \wt V \rightarrow \ol V$, and partitions $W'=\{W_{\ol v}' : \ol v \in \ol V\}$ and $W=\{W_{\wt v} : \wt v \in \wt V\}$ of $V$.
Let
$\mathcal H= (T,P)$ be the directed graph associated to $\mathcal
F$. We assume that $\mathcal H$ is strongly connected.  Let $p$ be
a finite path in ${\mathcal H}$. Then $p$ can correspond
to several words in $V^{+}= \{v_{1},\ldots v_{d}\}^{+}$:
for example if $p$ starts at vertex $[\ol v, \wt v]$,
then it corresponds to words starting with
$v$ whenever $v\in [\ol v, \wt v] .$ If $w$
 is a word in $V^{+}$ then we write $w=\ldots v$ to
mean that $w$ ends with $v$, and $w = v \ldots$ to mean that $w$ starts
with $v$.
It is not difficult to find  words $w\, \in V^{+}$  corresponding to a path in $\mathcal H$ such that
\begin{enumerate} \label{periodic_1}
\item
$w$ contains all $v_{i}$'s,
\item
$w^{2}$ corresponds to a legitimate path in $\mathcal H$, and
\item
for each $\wt v \in \wt V$, if $\sigma(\wt v) = \ol v$,  there exist words $p(\wt v)=\ldots \wt v$ and $s(\ol v)=\ol v \ldots $ such that
$w= p(\wt v)\, s(\ol v)$.
\end{enumerate}
Call  a word which satisfies (1) - (3) $\sigma$-{\em decomposable}.
If $w$ is a word, let $\overrightarrow{w}$ be the $d$-dimensional vector whose $i$-th entry is the number of occurrences of $v_{i}\, \in V$.

The following result generalizes Proposition \ref{d max/min paths},
and gives the constraints on the sequence $(F_{n})$ of transition
matrices that a diagram $B$ has in order for $B$ to support an
odometer with a periodic language.

\begin{proposition}
\label{odometers}
Let $B$ be a simple, strict rank $d$ Bratteli diagram.
Suppose that  $\mathcal F$ is a skeleton such that the associated graph $\mathcal H$ is strongly connected,
and let $w$ be a $\sigma$-decomposable word which corresponds to a path in $\mathcal H$.  Let $\{p_{v}^{(n)}\}_{v \in V,
n\geq 1}$ be a set of nonnegative integers.
If the
incidence matrices $(F_{n})$ of $B$ are such that the $v$-th row of $F_n$ is
\begin{equation}\label{row_structure}
\overrightarrow{s(\ol v)} + p_{v}^{(n)}\overrightarrow{w} + \overrightarrow{p(\wt v)}
\end{equation}
  whenever $v\in [\ol v , \wt v]$; then $(X_{B}, \varphi_{\om})$ is topologically conjugate to an
 odometer.
\end{proposition}

\begin{proof}
Define, for $v\in [\ol v, \wt v]$, $w(v,n-1,n):= s(\ol
v)w^{p_{v}^{(n)}} p(\wt v)$. Note that the $v$-th row of $F^{(n)}$
is (\ref{row_structure}), and $(B, \om)$ has skeleton
$\mathcal F$. Now $\mathcal H$ tells us what words of length 2 are
allowed in $\mathcal L_{B,\om}$: $vv' \in \mathcal L_{B,\om}$  only if
$v\in [\ol v,\wt v]$, $v'\in [\ol v', \wt v']$, and
$\sigma (\wt v) = \ol v'$. Thus
\[ w(v,n-1, n)w(v',n-1,n) = s(\ol
v )w^{p_{v}^{(n)}} p(\wt v)\,\,\, s(\ol v')w^{p_{v'}^{(n)}}
p(\wt v') = s(\ol v) w^{p_{v}^{(n)}} w w^{p_{v'}^{(n)}} P_{\wt
v'}\,\] by property (3) of a $\sigma$-decomposable word. Since
$w(v,n-1,n+1)$ (and more generally, $w(v, n-1, N)$) is a concatenation
of words $w(v,n-1,n)$, this implies that ${\mathcal L}_{B,\om}$ is periodic.
Proposition \ref{adic_odometers} implies the desired result.

\end{proof}

There is a converse to this result:
namely that if a perfect order $\om$ on  a simple diagram $B$ has a periodic language, then there is
some $\sigma$-decomposable word which generates $\mathcal L(B,\om)$, so
that by Lemma \ref{adic_odometers}, $(X_{B},\varphi_{\om})$ is an odometer.

If  $V = \{v_{1}, v_{2}, \ldots v_{d}\}$ and a perfect $\om$ is to have
$d$ maximal paths, then Proposition \ref{d max/min paths} tells us that
$v_{1}v_{2}\ldots v_{d}$ is, up to rotation, the only
$\sigma$-decomposable word.
 The next example shows that in general $\sigma$-decomposable words are easy to find.

\begin{example}
Let $V = \{a_{1}, a_{2}, \ldots a_{n+1}\}$, $\ol V = \wt V =
\{a_{1},a_{2}, \ldots a_{n}\}$, $\sigma (a_{i}) = a_{i+1}$ for
$i<n$ and $\sigma (a_{n}) = a_{1}$, where $[a_i, a_i]= \{a_{i}\}$ for each $i$ and  $a_{n+1} \, \in [a_i, a_j]$ for
some $j\neq i$. Then any word starting with $a_{i}$ (for $1 \leq i
\leq n$), ending with $\sigma^{-1}(a_{i})$, and containing all $a_{i}$'s
 is $\sigma$-decomposable.

\end{example}

\section{A characterization of finite rank diagrams that support perfect, non-proper orders}\label{characterization}

In this section, which is built on the results  of Section \ref{language_skeleton}, we discuss  the question of under what conditions a simple
 rank $d$ Bratteli diagram $B$ can have a perfect ordering $\om$ belonging to $\mathcal O_B(k)$ for
$ 1<  k \leq d$. It turns out
that the incidence matrices must satisfy certain conditions, which in turn depend on the skeleton that one is considering.

Let $(B,\om)$ be a perfectly  ordered simple Bratteli diagram. We continue to assume that $(B,\om)$ is well telescoped.
Let $\mathcal F = \mathcal F_\om$ be the skeleton generated by $\om$ and  let $\varphi = \varphi_\om$ be the corresponding Vershik
map. We have $|\wt V| = |\ol V|$ and $\varphi_\om$ defines a
one-to-one map $\sigma : \wt V \to \ol V$ such that $\varphi_{\om} (M_v) = m_{\sigma(v)}$ for $ v \in \wt V$.
Recall also the two partitions $W = \{W_{\wt v} : v \in \wt V\}$ and $W' = \{W'_{\ol v} : v \in \ol V\}$ of $V$ generated by $\mathcal F$.

We need some new notation. Recall that we write $\wt V_n$ ($\ol V_n$) instead of just $\wt V$ ($\ol V$) if we need to specify  in which level $\wt V$ ($\ol V$) lies. Let $E(V_n, u)$ be the set of all finite paths
between vertices of level $n$ and a vertex $u\in V_m$ where $m >
n$. The symbols $\wt e(V_n, u)$ and $\ol e(V_n, u)$ are used to denote
the maximal and minimal finite paths in $E(V_n, u)$, respectively. By $\wt V_n$ we mean that we are looking at the set $\wt V$ of vertices at level $n$.
Fix
maximal and minimal vertices $\wt v$ and $\ol v$ in $\wt V_{n-1}$ and $\ol V_{n-1}$ respectively. Denote $E(W_{\wt
v}, u) = \{e \in E(V_n, u): s(e) \in W_{\wt v}, r(e) = u\}$ and
$\wt E(W_{\wt v}, u) = E(W_{\wt v}, u) \setminus \{\wt e(V_n,
u)\}$. Similarly, $\ol E(W'_{\ol v}, u) = E(W'_{\ol v}, u)
\setminus \{\ol e(V_n, u)\}$. Clearly, the sets $\{E(W_{\wt v}, u)
: \wt v \in \wt V\}$ form a partition of $E(V_n, u)$.
Let $e$ be a non-maximal finite path, with  $r(e)=v$ and $s(e)\in V_m$,  which determines the cylinder set $U(e)$.  By $\varphi_\om(e)$ we mean $\varphi_\om (U(e))$, the image, under $\varphi_\om(e)$, of the cylinder set $U(e)$, which also has range $v$, and source in $V_m$.

\begin{lemma}\label{successor} Let $(B, \om)$ be a perfectly  ordered, well telescoped finite rank simple diagram, where $\om$  has  skeleton  $\mathcal F_\om$ and permutation $\sigma:\wt V \rightarrow \ol V$.
If $n> 1$, $\wt v\in \wt V_{n-1}$ and  $u\in V_m\ (m> n)$, then for any finite path $e \in \wt E(W_{\wt v}, u)$ we have $\varphi_\om(e) \in \ol E(W'_{\sigma(\wt v)}, u)$.
\end{lemma}

\begin{proof}
Note that $s(e)s(\varphi_\om(e))$ is a subword of $w(u,n,m)$. Now $s(e) \in W_{\wt v}$ by assumption and $s(\varphi_\om(e))\in W_{\ol v}$ for some $\ol v$.
This implies that $\wt v \ol v$ is a sub word of $w(u,n-1,m)$. Recalling that $(B,\om)$ is telescoped,  the result follows.

\end{proof}

We immediately deduce from the previous lemma that the following result on entries of incidence matrices is true.
 \begin{corollary}\label{necessary condition}
In the notation of Lemma \ref{successor}, the following condition holds for the perfectly ordered, well telescoped  finite rank  simple diagram $(B, \om)$:  for any $n \geq 2$, any vertex $\wt v \in \wt V_{n-1}$,  $m >n$, and any $u \in V_m$ one has
$$
|\wt E(W_{\wt v}, u)| = |\ol E(W'_{\sigma(\wt v)}, u)|.
$$
\end{corollary}

In particular, if $B$ is as above, and $(F_n) = ((f_{v,w}^{(n)}))$
denotes the sequence of positive
incidence matrices for $B$, then we can
apply Corollary \ref{necessary condition} to obtain the following
property on $F_n$. Define two sequences of matrices $\wt F_n = (\wt
f_{w,v}^{(n)})$ and $\ol F_n = (\ol f_{w,v}^{(n)})$ by the following
rule (here $w\in V_{n+1}$, $ v\in V_n$ and $ \ n\geq 1$):
\begin{equation}\label{wt f}
\wt f_{w,v}^{(n)} = \left\{
                      \begin{array}{ll}
                        f_{w,v}^{(n)} - 1, & \mbox{ if }\wt e_w \in E(v, w); \\
                        f_{w,v}^{(n)}, & \mbox{otherwise},
                      \end{array}
                    \right.
\end{equation}
\begin{equation}\label{ol f}
\ol f_{w,v}^{(n)} = \left\{
                      \begin{array}{ll}
                        f_{w,v}^{(n)} - 1, & \mbox{ if }\ol e_w \in E(v, w); \\
                        f_{w,v}^{(n)}, & \mbox{otherwise}.
                      \end{array}
                    \right.
\end{equation}
Then for any $u\in V_{n+1}$ and $ \wt v\in \wt V_{n-1}$, we obtain that under
the conditions of Corollary \ref{necessary condition} the entries of incidence matrices have the property:
\begin{equation}\label{sum condition1}
\sum_{w\in W_{\wt v}} \wt f_{u,w}^{(n)} \ \  = \sum_{w'\in W'_{\sigma(\wt v)}} \ol f_{u,w'}^{(n)},\ \ n \geq 2.
\end{equation}
We call  relations  (\ref{sum condition1}) the {\em balance relations}.

Given  $(\mathcal F, \sigma)$ on $B$, is it sufficient for $B$ to satisfy the balance relations so that there is a perfect order on $B$ with associated skeleton and permutation $(\mathcal F, \sigma)$? Almost. We need one extra condition on $B$.
First
  we need finer notation for $\mathcal H$: we replace it with a sequence $(\mathcal H_n)$ where each $\mathcal H_n$ looks exactly the same as $\mathcal H$, except that the vertices $T_n$ of $\mathcal H_n$ are labeled $[\ol v, \wt v, n]$. Paths in $\mathcal H_n$ will correspond to words from $V_n$, in particular, the word $w(u,n,n+1)$ will correspond to a path in $\mathcal H_n$. (In the case where $B$ is a stationary diagram, there is no need to replace $\mathcal H$ with $(\mathcal H_n)$.)

   \begin{definition}Fix $n\in \N$ and $u\in V_{n+1}$. If
  $[\ol v, \wt v,n] \in \mathcal H_n$, we define  the {\em crossing number} $P_u([\ol
    v, \wt v,n])$ for the vertex $[\ol v ,\wt v,n]$ as
\[ P_u([\ol
  v, \wt v,n]) :=\sum_{w\in[\ol v,\wt v, n]}\wt f_{uw}^{(n)} . \]
   \end{definition}
  This crossing number represents the
  number of times that we will have to pass through the vertex $[\ol
    v,\wt v,n]$ when we define an order on $r^{-1}(u)$, for $u\in V_{n+1}$, and here we
  emphasize that if we terminate at $[\ol v,\wt v,n]$ , we do not
  consider this final visit as contributing to the crossing number -
  this is why we use the terms $\wt f_{u,w}^{(n)}$, and not
  $f_{u,w}^{(n)}$.
  \begin{definition}
  We say that $\mathcal H_n$ is {\em positively strongly
    connected} if for each $u\in V_{n+1}$, the set of vertices $\{ [\ol v,
    \wt v,n]: P_u([\ol v,\wt v,n])>0 \}$, along with all the relevant
  edges of $\mathcal H_n$, form a strongly connected subgraph of
  $\mathcal H_n$.
  \end{definition}
  If $s(\wt e_u) \in [\ol v, \wt v,n]$ we shall call  this
  vertex in $\mathcal H_n$ the {\em terminal vertex}  (for $u$), as when defining
  the order on $r^{-1}(u)$, we need a path that ends at this vertex
  (although it can previously go through this vertex several times - in
  fact precisely $P_u([\ol v, \wt v,n])$ times).

\begin{example} In this example we have a stationary diagram so we  drop the dependance on $n$.
Suppose that $V=\{a,b,c,d\}$, $\ol V = \wt V = \{a,b,c\}$, with $a \in [a,a]$, $b\in [b,b]$, $c\in [c,c]$ and $d\in [b,a]$. Let $\sigma(a)=b$, $\sigma(b)=c$ and $\sigma( c )=a$. Suppose that for each $n\geq 1$ the incidence matrix $F=F_n$ is
\[F:=
\left(
\begin{array}{cccc}
2& 1 & 1 &  1 \\
1& 2 & 1  & 1 \\
1 & 1 & 2  & 1 \\
1  & 1  &1 & 2  \\
 \end{array}
\right)
\]
Then if $u=d$, $P_d([a,a])=0$, and the remaining three vertices $[b,b]$, $[c,c]$ and $[b,a]$ do not form a strongly connected subgraph of $\mathcal H$:   there is no path from $[c,c]$ to $[b,a]$.

Note also that although the rows of this incidence matrix satisfy the balance relations  (\ref{sum condition1}), there is no way to define an order on $r^{-1}(d)$ so that the resulting global order is perfect. The lack of positive strong connectivity of the graph $\mathcal H$ is precisely the impediment.
\end{example}

The following result shows that, given a skeleton $\mathcal F$ on $B$, as long as the associated graphs $(\mathcal H_n)$ are eventually positively strongly connected,  the balance relations are sufficient to define a perfect ordering $\om$ on a simple Bratteli diagram.

\begin{theorem}\label{existence of good order}
Let $B$ be a simple strict rank $d$ Bratteli diagram,  let $\mathcal F = \{M_{\wt v}, m_{\ol v}, \wt e_w, \ol e_w: w \in V^* \backslash V_0,\,\, \wt v \in \wt V\mbox{  and } \ol v \in \ol V  \}$ be a skeleton on $B$, and let $\sigma : \wt V \to \ol V$ be a bijection.  Suppose that eventually all associated graphs $\mathcal H_n$ are positively strongly connected, and suppose that the entries of incidence matrices $(F_n)$ eventually satisfy the balance relations (\ref{sum condition1}).
Then there is a perfect ordering $\om$ on $B$ such that $\mathcal F = \mathcal F_\om$ and the Vershik map $\varphi_\om$ satisfies the relation $\varphi_\om(M_{\wt v}) = m_{\sigma(\wt v)}$.
\end{theorem}

\begin{proof}  Fix $n$ large enough so that $\mathcal H_n$ is positively strongly connected and the balance relations hold.
Our goal is to define a linear order $\om_u$ on $r^{-1}(u)$ for each $u\in V_{n+1}$. Once this is done for all $n$ large,  the corresponding partial ordering $\om$ on $B$ will be  perfect. Recall that each set  $r^{-1}(u)$ contains two pre-selected edges $\wt e_u, \ol e_u$ and they should be the maximal and minimal edges after defining $\om_u$.

Fix $u \in V_{n+1}$. The proof is based on an recursive procedure that is applied to the $u$-th  row of the incidence matrix $F_n$. We  describe in detail the first step of the algorithm that will be applied repeatedly. At the end of each step in the algorithm, one entry in the $u$-th  row of  $F_n$ will have  been reduced by one, and a path in $\mathcal H_n$ will have been extended by one edge. At the end of the algorithm, the $u$-th row will have been reduced to the zero row, and a path will have been constructed in $\mathcal H_n$, starting at the vertex in $\mathcal H_n$ to which $s(\ol e_u)$ belongs, and ending at the vertex in $\mathcal H_n$ to which $s(\wt e_u)$ belongs. This path will determine the word $w(u,n,n+1)$, i.e. the order $\om_u$ on $r^{-1}(u)$.
It will be seen from the proof of the theorem that for given $\mathcal F$ and $\sigma$, the order $\om_u$ that is defined  is not unique.

We will first consider the particular case when the associated graph $ \mathcal H_n$ defined by $(\mathcal F, \sigma)$ does not have any loops. After that, we will modify the construction to include possible loops in the algorithm. We also include Examples \ref{3 max/min in rank 6} and \ref{looped_example} to illustrate why it is necessary to consider these cases.

 Case I: there is no loop in $\mathcal H_n$. Consider the $u$-th rows of matrices $\ol F_n$ and
 $\wt F_n$. They coincide with the row  $(f_{u,v_1}^{(n)},\ldots , f_{u, v_d}^{(n)})$ of the matrix
 $F_n$ except only one entry  corresponding to $|E(s(\ol e_u),
 u)|$ and one entry corresponding to  $|E(s(\wt e_u), u)|$. To simplify our notation, since $n$ is fixed we omit it
 as an index, so that $F=F_n$, $f_{u,w}= f_{u,w}^{(n)}$, $[\ol v, \wt v]=[\ol v, \wt v,n]$, $\mathcal H = \mathcal H_n$, etc.

 Take $\ol e_u$ and
 assign the number $0$ to it, i.e., $\ol e_u$ is the minimal edge in
 $\om_u$. Let $[\ol v_0, \wt v_0] $ be the vertex\footnote{The same word
 `vertex' is used in two meanings: for elements of the set $T$ of the
 graph $\mathcal H$ and for elements of the set $V$ of the Bratteli
 diagram $B$. To avoid any possible confusion, we point out explicitly
 what vertex is meant in that context.} of $\mathcal H$ such that
 $s(\ol e_u) \in [\ol v_0, \wt v_0]$. Consider the set $\{ \wt v
 \in \wt V : [\sigma(\wt v_0), \wt v] \in \mathcal H\}$ (this set is
 formed by ranges of arrows in $\mathcal H$ coming out from $[\ol
 v_0, \wt v_0]$).  Find $w'$ such that $\wt f_{u, w'} \geq \wt f_{u, w}$ for all
 entries $f_{u, w}, w \in W'_{\sigma(\wt v_0)}$. If there are several
 entries that are the maximal value, then $f_{u, w'}$ is chosen arbitrarily
 amongst them. Take any edge $e_1 \in E(w', u)$. In the case where  $\wt e_u \in
 E(w', u)$, we choose $e_1 \neq \wt e_u$. Assign the number $1$ to $e_1$
 so that $e_1$ becomes the successor of $e_0= \ol e_u$. We note also that the choice of $w'$ from $W'_{\sigma(v_0)}$ actually
 means that we take some $\wt v_1\in \wt V$ such that $s(e_1)\in
 [\sigma(\wt v_0), \wt v_1]$.
 In other words, we take the edge from $[\ol v_0, \wt v_0]$ to $[\sigma(\wt v_0), \wt v_1]$ in the associated graph $\mathcal H$.

We note that in the collection of relations (\ref{sum condition1}),
numerated by vertices from $\wt V$, we have worked with the equation
defined by $u$ and $\wt v_0$. Two edges were labeled in the above
procedure, $e_0$ and $e_1$. We may think of this step as if these
edges were `removed' from the set of all edges in $r^{-1}(u)$. We claim  that
the remaining non-enumerated edges satisfy the equation
\begin{equation}\label{sum condition for v_0}
(\sum_{w\in W_{\wt v_0}} \wt f_{u,w})  -1 \   = (\sum_{w\in W'_{\sigma(\wt v_0)}} \ol f_{u,w})  -1.
\end{equation}

To see this, note that
 $\wt v_1 \neq  \wt v_0$: for if not, then $\sigma(\wt v_1)=\sigma(\wt
 v_0)$, but this implies  that there would be a loop at $[\sigma(\wt
 v_0), \wt v_1]$, a contradiction to our assumption.
 Thus $\wt v_1\neq\wt
 v_0$ and this is why there is exactly one edge removed from each side
 of  (\ref{sum condition for v_0}).
 Note that we now have a  `new', reduced $u$-th row of $F$. Namely, the entry $f_{u,\ol v_0}$ has been reduced by one. Thus the crossing numbers of the vertices of $\mathcal H$ change (one crossing number is reduced by one). Also note that in $\mathcal H$, we have arrived at the vertex $[\sigma(\wt v_0), \wt v_1]$ to which $w'$ belongs. Thus for this reduced $u$-th row,
 $\ol {f}_{u,w'}= f_{u,w'} -1$. In other words,
 with each step of this algorithm the row we are working with changes, and the vertex $w$ such that $\ol f_{u,w}= f_{u,w} -1$ changes (in fact, has to change, because there are no loops in $\mathcal H$).  For, the vertex such that $\ol f_{u,w}= f_{u,w} -1$  belongs to the vertex in $\mathcal H$ where we  are currently, and this changes at every step of the algorithm.
 Thus the new reduced $u$-th row of $F$ still satisfies the balance relations (\ref{sum condition1})
as  $\wt v \in \wt V$ varies.
 This completes the first step of the construction.

We apply the described procedure again to show how we should proceed to complete the next step.
Let us assume that all crossing numbers ares still positive for the time being to describe the second step of the algorithm.

 Consider the set $\{f_{u,w} : w \in W'_{\sigma(\wt v_1)}\}$ and find some $w''$ such that $\wt f_{u,w''} \geq \wt f_{u,w}$ for any $w\in W'_{\sigma(\wt v_1)}$. In the corresponding set of edges $E(w'', u)$ we choose $e_2 \neq \wt e_u$, and assign the number $2$ to the edge $e_2$, so that $e_2$ is the successor of $e_1$.

Observe that now we are dealing with the relation of (\ref{sum
  condition1}) that is determined by $\wt v_1\in \wt V$. If we again
`remove' the enumerated edges $e_1$ and $e_2$ from it, then this
relation remains true with both sides reduced by $1$ as we saw the
same in (\ref{sum condition for v_0}).

We remark also that the choice that we made of $w''$ (or $e_2$) allows
us to continue the existing path (in fact, the edge) in $\mathcal H$
from $[\ol v_0, \wt v_0]$ to $[\sigma(\wt v_0), \wt v_1]$ with the edge from $[\sigma(\wt v_0), \wt v_1]$ to $[\sigma(\wt
v_1), \wt v_2]$, where $\wt v_2$ is defined by the property that
$s(e_2) \in [\sigma(\wt v_1) ,\wt v_2]$.

This process can be continued. At each step we apply the following rules:

(1) the edge $e_i$, that must be chosen next after $e_{i-1}$, is taken
from the set $E(w^*, u)$ where $w^*$ corresponds to a maximal entry
amongst $\wt f_{u,w}$ as $w$ runs over $W'_{\sigma(\wt v_{i-1})}$;

(2) the edge $e_i$ is always taken not equal to $\wt e_u$ unless no more
 edges except $\wt e_u$ are left.

After every step of the construction, we see that the following statements hold.

(i) Relations (\ref{sum condition1}) remain true when we treat them as the number of non-enumerated edges left in $r^{-1}(u)$. In other words, when a pair of vertices $\wt v$ and $\sigma(\wt v)$ is considered, we reduce by 1 each side of the equation defined by $\wt v$.

(ii) The used procedure allows us to build a path $p$ from the starting
vertex $[\ol v_0, \wt v_0]$ going through other vertices of the graph
$\mathcal H$ according to the choice we make at each step. We need to guarantee that at each step, we are able to move to a vertex in $\mathcal H$ whose crossing number is still positive (unless we are at the terminal stage). As long as the crossing numbers of vertices in  $\mathcal H$ are positive, there is no concern. Suppose thought that we land at a (non-terminal) vertex $[\ol v, \wt v]$ in $\mathcal H$  whose crossing number is one (and this is the first time this happens). When we leave this vertex, to go to $[\sigma(\wt v), \wt v' ]$,  the crossing number for $[\ol v, \wt v]$ will become 0 and therefore it will no longer be a vertex of $\mathcal H$ that we can `cross' through, maybe only arriving at it terminally. {\em Thus at this point, with each step, the graph $\mathcal H$ is also changing (being reduced).} We need to ensure that there is a way to continue the path out of $[\sigma(\wt v), \wt v' ]$.
 Since $\sum_{w\in W_{\wt v} } \wt f_{u,w}\geq      P_u[\sigma(\wt v), \wt v ]= 1$, then by the balance relations, $\sum_{w' \in W_{\sigma(\wt v')} } \ol f_{u,w'}\geq 1$. If the crossing number of all the vertices $[\sigma(\wt v'), *]$ have been reduced to 0, then this means that for a unique  $w'$, $ \ol f_{u,w'}= 1$ (the rest of the summands in $\sum_{w' \in W_{\sigma(\wt v')} } \ol f_{u,w'}$ equal 0),
 and $\wt f_{u,w'}= 1$:
 this tells us that we have to move into this terminal vertex for the last time.
Then the balance equations, which continue to be respected, ensure we are done. Otherwise,
 the balance equations guarantee that $\sum_{w' \in W_{\sigma (\wt v')}} \ol f_{u,w'}> 1$, which means there is a valid continuation of our path out of $[\sigma(\wt v), \wt v' ]$ and to a new vertex in $\mathcal H$, and we are not at the end of the path. It is these balance equations which always ensure that the path can be continued until it reaches its terminal vertex.

(iii) In accordance with (i), the $u$-th row of $F$ is transformed by a sequence
of steps in such a way that entries of the obtained rows form decreasing
 sequences. These entries show the number of non-enumerated edges
 remaining  after the completed steps. It is clear that, by the rule used above, we decrease the largest entries first. It follows from the simplicity of the diagram that, for sufficiently many steps, the set $\{s(e_i)\}$ will contain all vertices $v_1,..., v_d$ from $V$. This means that the transformed $u$-th row consists of entries which are strictly less than those of the very initial $u$-th row $F$. After a number of steps the $u$-th row will have a form where the difference between any two entries is $\pm 1$. After that, this property will remain true.

(iv) It follows from (iii) that we finally obtain that all entries of
 the resulting  $u$-th row are zeros or ones. We apply the same procedure to enumerate the remaining edges from $r^{-1}(u)$ such that the number $|r^{-1}(u)| - 1$ is assigned to the edge $\wt e_u$. This means that we  have constructed the word $W_u = s(\ol e_u)s(e_1) \cdots s(\wt e_u)$, i.e. we have ordered $r^{-1}(u)$.

Looking at the path $p$ that is simultaneously built in $\mathcal H$, we
 see that the number of times this path comes into and leaves  a vertex $[\ol v, \wt v]$ of the graph is precisely the crossing number of $[\ol v, \wt v]$ . In other words, $p$ is an Eulerian path of $\mathcal H$ that finally arrives to the vertex of $\mathcal H$ defined by $s(\wt e_u)$.

Case II: there is a loop in $\mathcal H = \mathcal H_n$.
To deal with this case, we
have to refine the described procedure to avoid a possible situation
when the algorithm cannot be finished properly.

We start as in Case (I), and continue until we have arrived to a vertex $[\ol v_1, \wt v]$, where, for the first time, $[\sigma(\wt v), \wt v] \in \mathcal H$.  In other words, this is the first time that our path reaches a vertex which has a successor with a loop.
If $[\sigma(\wt v), \wt v] $
 has  crossing number zero, - i.e. it is the terminal vertex - and we are not at the terminal stage of defining the order, we ignore this vertex and continue as in Case (I). If $[\sigma(\wt v), \wt v] $
 has a positive crossing number, i.e.  $P_u (  [\sigma(\wt v), \wt v]  )>0  $,  then at this point, we continue the path to
  $[\sigma(\wt v), \wt v] $, and then
traverse this loop $P_u (  [\sigma(\wt v), \wt v]  )-1$ times. If $P_u (  [\sigma(\wt v), \wt v]  ) =    \sum_{w\in  [\sigma(\wt v), \wt v]  } \wt f_{u,w}=    \sum_{w\in  [\sigma(\wt v), \wt v]  } f_{u,w}$ this means we are traversing this vertex enough times that it is no longer part of the resulting $\mathcal H$ that we have at the end of this step - we are removing the looped vertex. If  $P_u (  [\sigma(\wt v), \wt v]  ) =   \sum_{w\in  [\sigma(\wt v), \wt v]  } \wt f_{u,w}=   ( \sum_{w\in  [\sigma(\wt v), \wt v]  } f_{u,w}) -1,$  then we are reducing this vertex to a vertex whose crossing number is 0 and we will only return to this vertex at the very end of our path.
 Looking at the relation
\begin{equation}\label{loop equation}
\sum_{w\in W_{\wt v}} \wt f_{u,w} \ \  = \sum_{w'\in W'_{\sigma(\wt v)}} \ol f_{u,w'},
\end{equation}
we see that we
have removed all the values $\wt f_{u,w}$, where $w \in [\sigma(\wt v), \wt v] $ on the left hand side, and also this same amount of values from the right hand side.
  We consequently
 enumerate  all edges whose source lies in $[\sigma(\wt v), \wt v] $ in {\em arbitrary
 order}.

 We also need to ensure that once we have traversed this loop the required number of times, we can actually leave this vertex $[\sigma(\wt v), \wt v] $. To see this,
we first make a remark about the graph $\mathcal H$. Suppose that there is a loop in $\mathcal H$ at $[\ol v, \wt v]$, whose crossing number is positive.
If $[\ol v_1, \wt v]$ is a (non-looped) vertex with a positive crossing number, which has $[\ol v, \wt v]$ (the vertex with the loop) as a successor, then for some $\wt v'\neq \wt v$, the vertex $[\ol v, \wt v']$ will satisfy
 $\sum_{w'\in [\ol v, \wt v']} \ol f_{u,w'} >0$.
 This is because of our discussion above concerning (\ref{loop equation}): the crossing number at the looped vertex appears on both sides, and cancel. So if $[\ol v_1, \wt v]$ has a positive crossing number, this contributes positive values to the left hand side of (\ref{loop equation}); thus there is some vertex $[\ol v, \wt v']$ with a positive value
 $\sum_{w'\in [\ol v, \wt v']} \ol f_{u,w'} $ contributing
 to the right hand side of (\ref{loop equation}). All this means that we are able to continue our path out of the looped vertex $[\sigma(\wt v), \wt v] $.

 Then we return to the procedure from (I), until  we reach a vertex with a looped vertex as a successor, and revert to the procedure from (I) when we have removed the looped vertex.

To summarize Cases I and II, we notice that, constructing the Eulerian
 path $p$, the following rule is used: as soon as $p$ arrives before a loop
 around a vertex $t$ in $\mathcal H$,  $p$ traverses this vertex $P_u(t)-1$ times.
 Then $p$
 leaves $t$ and goes to the vertex $t'$ according to the procedure  in
 Case I.

As noticed above, the fact that all edges $e$ from $r^{-1}(u)$ are enumerated is equivalent to defining a word formed  by the sources of $e$. In our construction, we obtain the word $w(u, n, n+1)= s(\ol e_u) s(e_1)\cdots s(e_j)\cdots s(\wt e_u)$.

Applying these arguments to every vertex $u$ at every sufficiently advanced level of the  diagram, we define
 an ordering $\om$ on $B$. That $\om$ is perfect follows from Lemma
 \ref{V-H-correspondence}: we chose $\om$ to have skeleton $\mathcal F$,
 and for each $n\geq 1$, constructed all words $w(v,n,n+1)$  to correspond to
 paths in $\mathcal H_n$. The result follows.

\end{proof}

\begin{remark}\label{non-simple case}
We observe that the assumption about simplicity of the Bratteli diagram in the above theorem is redundant. It was used only when we worked with strictly positive rows of incidence matrices. But for a non-simple finite rank diagram $B$ we can use the following result proved in \cite{bezuglyi_kwiatkowski_medynets_solomyak:2011}.

\textit{ Any Bratteli diagram of
finite rank is isomorphic to a diagram  whose incidence matrices $(F_n)$  are of the form
\begin{equation}\label{Frobenius Form: General}
F_n =\left(
  \begin{array}{ccccccc}
    F_1^{(n)} & 0 & \cdots & 0 & 0 & \cdots & 0 \\
    0 & F_2^{(n)} & \cdots & 0 & 0 & \cdots & 0 \\
    \vdots & \vdots & \ddots & \vdots & \vdots & \cdots& \vdots \\
    0 & 0 & \cdots & F_s^{(n)} & 0 & \cdots & 0 \\
    X_{s+1,1}^{(n)} & X_{s+1,2}^{(n)} & \cdots & X_{s+1,s}^{(n)} & F_{s+1}^{(n)} & \cdots & 0 \\
    \vdots & \vdots & \cdots & \vdots & \vdots & \ddots & \vdots \\
    X_{m,1}^{(n)} & X_{m,2}^{(n)} & \cdots & X_{m,s}^{(n)} & X_{m,s+1}^{(n)} & \cdots & F_m^{(n)} \\
  \end{array}
\right).
\end{equation}
For every $n\geq 1$, the matrices $F_i^{(n)},\ i=1,...,s $, have strictly positive entries and matrices $F_i^{(n)},\ i=s+1,...,m$, have either strictly
positive or zero entries. For every fixed $j = s+1,...,m$, there is
at least one non-zero matrix $X_{j,k}^{(n)}$}.

It follows from (\ref{Frobenius Form: General}) that, for $u\in
V_{n+1}$, the $u$-th row of $F_n$ consists of several parts such that
the proof of Theorem \ref{existence of good order} can be applied to
each of these parts independently. Indeed, it is obvious that if $u$
belongs to any subdiagram defined by $(F_i^{(n)}),\ i=1,..., s$, then
we have a simple subdiagram. If $u$ is taken from
$(F_i^{(n)}),\ i=s+1,..., m$, then by (\ref{Frobenius Form: General})
we may have some zeros in a row but they do not affect the procedure
in the proof of Theorem \ref{existence of good order}.
  \end{remark}
\medskip

We illustrate the proof of Theorem  \ref{existence of good order} with
the following examples.

\begin{example}\label{3 max/min in rank 6}
 Suppose $B$ is a rank 6 Bratteli diagram defined on the vertices
 $\{a,b,c, d, e, f\}$. Let $\ol V = \wt V = \{a,b, c\}$ and $\sigma(a) =
 b, \sigma(b) = c, \sigma(c) = a$.  Take the  skeleton $\mathcal F=
 \{M_a, M_b, M_c, m_a, m_b, m_c; \ol e_d, \wt e_d, \ol e_e, \wt e_e, \ol
 e_f, \wt e_f\}$ where $s(\ol e_d) = b,\ s(\ol e_e) = b,\ s(\ol e_f) =
 c$ and $s(\wt e_d) =a,\ s(\wt e_e) =a,\ s(\wt e_f) =c$. For simplicity
 of notation, we suppose that $B$ is  stationary.
 For such a choice of the data, we see that non-empty intersections of
 partitions $W$ and $W'$ give the following sets: $[a, a] =
 \{a\},\ [b , a] = \{d,e\},\ [b, b] = \{b\},\ [c , c]
 = \{c,f\}$. The graph $\mathcal H$ is illustrated in
 Figure \ref{third_pic}.

 \begin{figure}[h]
\centerline{\includegraphics[scale=1.1]{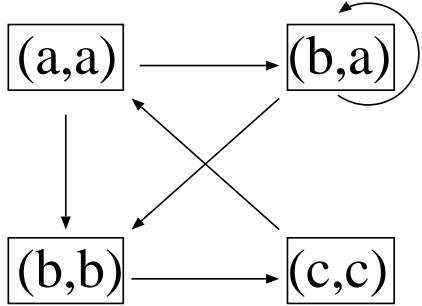}}
\caption{The graph associated to $\mathcal F_\omega$ in Example \ref{3 max/min in rank 6}
  \label{third_pic}}
\end{figure}

We see that $\mathcal H$ has four vertices and one loop around the vertex $[b,a]$. The directed edges are shown on the figure and defined by $\sigma$.

We consider, for definiteness, the case $u = a$ only and construct an
order on $r^{-1}(a)$ according to Theorem \ref{existence of good
  order}. In this case, the balance relations have the
form: $f_{a,a} -1 = f_{a,b} = f_{a,c} + f_{a,f}$ and the entries
$f_{a,d}, f_{a,e}$ can be taken arbitrarily because they correspond to
the loop in $\mathcal H$. For instance, the following row $(3, 2, 1,
3, 2, 1)$ satisfies the above condition.  Applying the algorithm
in the proof of  Theorem \ref{existence of good order}, we can order the edges from
$r^{-1}(a)$ such that their sources form the word $w(a, n-1, n) =
addeedbfabca$. To define an order on $r^{-1}(v), v= b,c,d,e,f$, we
apply similar arguments (details are left to the reader). By Theorem
\ref{existence of good order}, we conclude that if the entries of
incidence matrices satisfy (\ref{sum condition1}), then $B$ admits a
perfect ordering $\om $ such that $\mathcal F = \mathcal F_\om$ and
the Vershik map agrees with $\sigma$.
\end{example}

In the next example, we will show how one can describe the structure of
Bratteli diagrams of rank $d$ for which there exists a perfect ordering
with exactly $d-1$ maximal and minimal paths. The following example
deals with a finite rank 3 diagram.

\begin{example}\label{looped_example}
Suppose $B$ is a rank 3 diagram defined on the vertices $\{a,b,c\}$ with
 $\ol V = \wt V = \{a,b\}$ and $\sigma(a) = b, \sigma(b) = a$.  Take the
 skeleton $\mathcal F= \{M_a, M_b,  m_a, m_b;\ol e_d, \wt e_c, \ol
 e_c\}$ where $s(\ol e_c) = b,\ s(\wt e_c) = a$.
 For such a
 choice of the data, we see that $[a,a] = \{a\}, [a, b] =
 \emptyset, [b, a] = \{c\}, [b,b] = \{b\}$ and $\mathcal H$
 is illutrated in Figure \ref{figurea}.

\begin{figure}[h]\label{figurea}
\centerline{\includegraphics[scale=1.1]{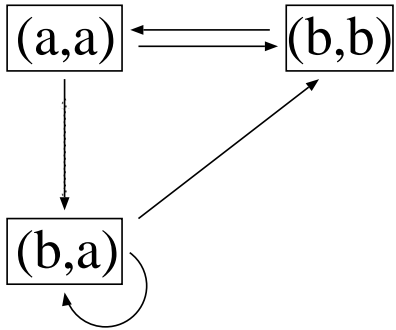}}
\caption{The graph associated to $\mathcal F$ in Example \ref{looped_example}
  \label{ordering_construction}}
\end{figure}

To satisfy the condition of Theorem \ref{existence of good order}, we have to take the incidence matrix
$$
F =\left(
     \begin{array}{ccc}
       f +1 & f & p \\
       g  & g+1 & q \\
       t & t  & s \\
     \end{array}
   \right)
$$
where the entries $f$, $g$, $p$, $q$ and $t$ are any positive integers. We note that the form of $F$ depends on the given skeleton. In order to see how Theorem \ref{existence of good order} works, one can choose some specific values for the entries of $F$ and repeat the proof of the theorem. For example, if the incidence matrix is of the form
$$
F =\left(
     \begin{array}{ccc}
       3& 2& 1 \\
       2 & 3 & 1 \\
       4 & 4 & 2\\
     \end{array}
   \right),
$$
then  one possibility for a valid ordering is  $w(a,n-1,n) = acbaba$, $w(b,n-1,n) = bacbab$ and $w(c,n-1,n) = baccbababa$. Note that there are other valid orderings that do not comply with our algorithm, for example $w(a,n-1,n) = abacba$.

Finally we show how looped vertices can cause trouble. Take the vector $(f+1,f,p)=(2,1,1)$ for the $a$-th row of $F$. Note that the only possible way to order $r^{-1}(a)$ is $r^{-1}(a) = acba$. In other words, the initial letter $a$ {\em must} be followed by the letter $c$: in our graph $\mathcal H$, we must go from the vertex $[a,a]$ to the looped vertex $[b,a]$, otherwise we cannot complete the ordering on $r^{-1}(a)$.

\end{example}

\section{The measurable space of orderings on a diagram}\label{genericity_results}

In this section we study $\mathcal O_B$ as a measure space.
Recall that $\mu = \prod_{v\in V^* \backslash V_0}\mu_v$ has been defined as the product measure on the set $\mathcal O_B= \prod_{v\in V^* \backslash V_0}P_v$, where each $\mu_v$ is the uniformly distributed measure on $P_v$. Also recall that
${\mathcal O}_{B}(j)$ is the set of orders on $B$ with $j$ maximal paths.

\begin{theorem}\label{goodbaddichotomy}
 Let $B$ be a finite rank $d$ aperiodic Bratteli diagram. Then there
 exists  $j \in \{1,..., d\}$ such that $\mu$-almost all orderings
 have $j$ maximal and $j$ minimal elements.

\end{theorem}

\begin{proof}
We shall first show that there exist $j$ and $j'$ such that  $\mu$-almost all orderings
 have $j$ maximal and $j'$ minimal elements. We then show that $j=j'$ in
 Corollary \ref{j=j'}.
If $B$ has rank $d$, then for $k \in \mathbb N$,  $1\leq i \leq d$ and $n>k$, define the event
\[G_{k}^{n,i} = \{ \om: \mbox{the  maximal paths from level $k$ to level $n$ have exactly $i$ distinct sources}\,\} , \]
and
\[ H_{k}^{i}: = \bigcup_{n> k} G_{k}^{n,i}\, . \]
We claim that
${\mathcal O}_{B}(1) = \limsup H_{k}^{1}$. For if $\om \in \limsup H_{k}^{1}$, then for some subsequence $(n_{k})$, $\om \in H_{n_k}^{1}= \bigcup_{n>n_{k}} G_{n_k}^{n,1} $ for each $k$. For each $n_{k}$, there is some $n>n_{k}$ such that the maximal paths from level $n_{k}$ to level $n$ have only one source. This means there is only one maximal path from level 1 to level $n_{k}$ that is extended to an infinite maximal path. Letting $n_{k}\rightarrow \infty$, we have that $\om \in \mathcal O_{B}(1)$.
Conversely, suppose that $\om \not\in \limsup H_{k}^{1}$. Then for some
$K$, and all $k>K$,
$$
\om \in (\bigcup_{n>k}G_{k}^{n,1})^{c}=
\bigcap_{n>k}\bigcup_{i=2}^{d}G_{k}^{n,i}.
$$
 Fix $k>K$. For some $j$, and
some $\{v_{1}\ldots v_{j}\}\subset V_{k}$, we have $\om \in
G_{k}^{n_{p},j}$ for infinitely many $n_p >k$, where the sources of
the maximal paths from level $k$ to level $n_p$ are $\{v_{1}\ldots
v_{j}\}$ for each of these $n_p$'s. Fix $n_1$; for some set
$\{v_{1}^{1},\ldots v_{j}^{1}\} \subset V_{n_1}$, and for some
subsequence $(n_{p^{(1)}})$ of $(n_{p})$, there are $j$ maximal paths
from level $k$ to level $n_{p^{(1)}}$ whose sources are $\{v_{1}\ldots
v_{j}\}$ and which pass through $\{v_{1}^{1},\ldots v_{j}^{1}\}
\subset V_{n_1}$, for any $n_{p^{(1)}}$. Let $\{M^{(i)}_{1}: 1\leq i
\leq d\}$ be the maximal paths from level k to level $n_1$ with
$r(M^{(i)}_{1}) = v_{i}^{1}$ for $1\leq i \leq j$.  Fix one $n_{2}$ from
$(n_{p^{(1)}})$. There exist $\{v_{1}^{2},\ldots v_{j}^{2}\} \subset
V_{n_2}$ and $(n_{p^{(2)}})$, a subsequence of $(n_{p^{(1)}})$, such
that for each $n_{p^{(2)}}$, there are $j$ maximal paths from level
$k$ to level $n_{p^{(2)}}$ with range $\{v_{1}^{2},\ldots v_{j}^{2}\}
\subset V_{n_2}$. Let $\{M^{(i)}_{2}: 1\leq i \leq d\}$ be the set of these
maximal paths. Each $M_{2}^{(i)}$ is  a refinement of $M_{1}^{(i)}$.
Continue in this fashion to get, for each $1\leq i\leq j$, a sequence
$(M_{j}^{(i)})$ of paths converging to $j$ distinct  maximal paths, so that
$\om \not\in \mathcal O_{B}(1)$.

Similarly we can show that for $1<j\leq d$,
\[ {\mathcal O}_{B}(j) = \left (\limsup_{k\rightarrow \infty}H_{k}^{j} \right )\backslash \bigcup_{i=1}^{j-1} {\mathcal O}_{B}(i)\, . \]

Now order the vertices in $V=\bigcup_{n\geq 1}
V_{n}$ as $\{v_{1},v_{2}, \ldots \}$ starting from level 2 and moving to
 levels $V_{n}$, $n=3,4, \ldots$.
for each $n\geq 1$ define the random variable $X_{n}$ on $\mathcal O_B$ where $X_{n}(w) = i$ if the source of the maximal edge with range $v_{n}$ is
the vertex $i$. The sequence $(X_{n})$ is a sequence of mutually
independent variables and if $\Sigma_{n}$ is the $\sigma$-field
generated by $\{X_{n},X_{n+1}, \ldots \}$ and $\Sigma:=\bigcap_n
\Sigma_{n}$, then for each $1\leq i \leq d$, ${\mathcal O}_{B}(j) \, \in
\Sigma$ and by Kolmogorov's zero-one law, for each $1\leq j \leq d$, $\mu( {\mathcal O}_{B}(j))$ is either 0 or 1.
Note now that one can repeat the definitions of all the above sets replacing the word `maximal' with `minimal'. The result follows.
\end{proof}

In the next result we use our notation from the proof of Theorem \ref{goodbaddichotomy}.

\begin{theorem}
\label{generic_three}
Let $B$ be an aperiodic Bratteli diagram of rank $d$.
\begin{enumerate}
\item $\mu ({\mathcal O}_{B}(1)) = 1$ if and only if there exists a sequence $(n_{k})_{k=1}^{\infty}$ such that
$\sum_{k=1}^{\infty} \mu(G_{n_{k}}^{n_{k+1}, 1}) = \infty$.
\item Let $1< j \leq d$. Then $\mu ({\mathcal O}_{B}(j)) = 1$ if and
only if there exists a sequence $(n_{k})$ where $\sum_{k}
\mu(G_{n_{k}}^{n_{{k+1}}, j}) = \infty$, and for each $1\leq i < j$,
and all sequences $(m_{k})$,  $\sum_{k}
\mu(G_{m_{k}}^{m_{{k+1}}, i}) < \infty$.
\end{enumerate}
\end{theorem}

\begin{proof}

(1)
Note that for each $j$ and  $n$ with $n>j$,
\begin{equation}\label{inclusion_0}
G_{j}^{n,1} \subset G_{j}^{n+1,1}
\end{equation}
 and similarly for each $j$, $n$ with $n>j+1$, $G_{j+1}^{n,1} \subset G_{j}^{n,1}$. This implies that
\begin{equation}\label{inclusion} H_{j+1}^{1} = \bigcup_{n>j+1}G_{j+1}^{n,1} \subset \bigcup_{n>j+1} G_{j}^{n,1} \subset
\bigcup_{n>j} G_{j}^{n,1} = H_{j}^{1}\, .\end{equation} If $\mu
({\mathcal O}_{B}(1)) = 1$, then since from the proof of Theorem
\ref{goodbaddichotomy}  ${\mathcal O}_{B}(1) =\limsup H_{k}^{1}$, we have
$$1=\mu ({\mathcal O}_{B}(1)) = \mu(\bigcap_{k=1}^{\infty}\bigcup_{j\geq k} H_{j}^{1})
\stackrel{(\ref{inclusion})}{=}
\mu(\bigcap_{k=1}^{\infty}H_{k}^{1}),
$$
which implies that for each $k$, $\mu(H_{k}^{1}) = 1$, and now inclusion (\ref{inclusion_0}) implies that for each $k$,
\begin{equation}\label{convergence}
1=\mu(H_{k}^{1})  = \mu ( \bigcup_{n>k}G_{k}^{n,1}) =\lim_{n\rightarrow \infty} \mu ( G_{k}^{n,1}),
\end{equation}
and this implies the existence of a sequence $(n_{k})$ such that  $\sum_{k=0}^{\infty}
\mu(G_{n_{k}}^{n_{{k+1}}, 1}) =\infty$.

Conversely, suppose there is some $(n_k)$ such that $\sum_{k} \mu(G_{n_{k}}^{n_{{k+1}}, 1}) =\infty$.
The converse of the Borel-Cantelli lemma implies that
for $\mu$-almost all
orderings, there is a subsequence $(j_{k})$  such that
all  maximal  edges in $E_{j_{k'}}$ have the same source. This
implies that for almost all  $\om$ there is at most one, and thus
exactly one maximal path in $X_{B}$.

(2) We prove Statement (2) for $j=2$, other cases follow similarly.
If $\mu({\mathcal O}_{B}(2))=1$, then $\mu({\mathcal O}_{B}(1))=0$, and  by the proof of Theorem \ref{goodbaddichotomy}, this means that
$\mu(\limsup H_{k}^{2})=1$ and $\mu(\limsup H_{k}^{1})=0$. Using
 (1), we conclude that for all sequences
$(m_{k})$,  $\sum_{k} \mu(G_{m_{k}}^{m_{{k+1}},
1}) < \infty$. Also, as in the proof of (1), we will have that for each $k$, $$\lim_{n\rightarrow \infty }\mu(G_{k}^{n,1}) = 0.$$
Note that for all $n>j$,
\begin{equation}\label{inclusion_4}G_{j}^{n,2} \subset
G_{j}^{n+1, 2} \cup G_{j}^{n+1, 1}\end{equation} and for all $n>j+1$,
$G_{j+1}^{n,2} \subset G_{j}^{n, 2} \cup G_{j}^{n, 1}$. This implies
that
\begin{equation}\label{inclusion2} H_{j+1}^{2} = \bigcup_{n>j+1}G_{j+1}^{n,2} \subset
\bigcup_{n>j+1} (G_{j}^{n,2} \cup  G_{j}^{n,1})
\subset\bigcup_{n>j} (G_{j}^{n,2} \cup G_{j}^{n,1}) = H_{j}^{2}\cup  H_{j}^{1}\, .
\end{equation}
It follows that $H_{n}^{2}\subset H_{j}^{2}\cup H_{j}^{1}$ whenever $n>j$.
As in Part (1) we have
$$1= \mu(\limsup H_{k}^{2}) \stackrel{(\ref{inclusion2})}{\leq }
\mu(\bigcap_{k=1}^{\infty}(H_{k}^{2}\cup H_{k}^{1})),
$$ so that for all $k$, $\mu(H_{k}^{2} \cap H_{k}^{1})=1$, and using Inclusion (\ref{inclusion_4}), this implies that $\lim_{n\rightarrow \infty} \mu(G_{k}^{n,2}\cup G_{k}^{n,1})=1$, so that
$\lim_{n\rightarrow \infty} \mu(G_{k}^{n,2})=1$. Now one can construct a suitable sequence $(n_{k})$ as was done in (1).

Conversely, if for  some $(n_k)$, $\sum_{k} \mu (
G_{n_k}^{n_{k+1}, 2 })$ diverges, then the converse of the Borel-Cantelli lemma implies that
almost all orders $\om$ have at most 2 maximal paths.
  Since for each sequence $(m_{k})$, $\sum_{k}
\mu(G_{m_{k}}^{m_{k+1}, 1}) < \infty$, Part (1) tells us that
$\mu(\mathcal O_{B}(1)) = 0$. The result follows.
\end{proof}

If $(F_{n})$, where $F_{n}= (f^{(n)}_{v,w}),$ is the sequence of  incidence matrices for $B$, consider the Markov matrices
  $M_{n}:= (m^{(n)}_{v,w})$ where
  $m^{(n)}_{v,w}:=\frac{f^{(n)}_{v,w}}{\sum_{w}f^{(n)}_{v,w}}$. Here
  $m^{(n)}_{v,w}$ represents the proportion of edges with range $v\in
  V_{n+1}$ that have source $w\in V_{n}$. Similarly, if $(n_k)$ is a given sequence,  consider for $j\geq 1$
\begin{equation}
\label{markov_telescoping}
F_{j}':= F_{n_{j+1}-1}\cdot F_{n_{j+1}-2} \cdot \ldots \cdot F_{n_{j}+1}
\end{equation}
and define the Markov matrices $M_{j}' = (m^{\prime (j)}_{v,w})$ as before.
Proposition \ref{generic_three} tells us that the integer $j$ such that
$\mu(\mathcal O_{B}(j))=1$
depends only on the masses of the sets $G_{n_k}^{n_{k+1},j}$, as $j$ and
$(n_k)$ vary. In turn, $\mu(G_{n_k}^{n_{k+1},j})$ depends only on the
matrices $M_{k}'$ where $F_{k}'$ is defined as in
(\ref{markov_telescoping}), and is independent of the word `maximal'
which was used to define the sets  $G_{n_k}^{n_{k+1},j}$.  We have
shown:

\begin{corollary}\label{j=j'}
Let $B$ be a finite rank $d$ aperiodic Bratteli diagram. For $\mu$-almost all orders $\om$, $|X_{\max}(\om)|=|X_{\min}(\om)|$.

\end{corollary}

The following corollary  gives a sufficient condition for diagrams $B$ where $\mu (\mathcal O_{B}(1)) = 1$.
Note that this case includes all simple $B$ with a bounded
number of edges at each level. We use the notation of relation
(\ref{markov_telescoping}).

\begin{corollary}
\label{generic_one}
Let $B$ be a Bratteli diagram with incidence matrices $(M_{n})$. Suppose there is some $\varepsilon >0$, sequences
 $(n_{k})$ of levels and   $(w_{k})$ of vertices (where $w_{k}\in
  V_{n_{k}}$), such that $m'\,^{(k)}_{v,w_{k}} \geq \varepsilon$ for all
$k\,\in \mathbb N$ and $v \in V_{n_{k+1}}$.
Then  $\mu (\mathcal O_{B}(1)) = 1$.
\end{corollary}

\begin{proof}
The satisfied condition implies that $\mu(G_{n_k}^{n_{k+1},1})
 \geq\varepsilon^{d}$.
Now apply
Proposition \ref{generic_three}.
\end{proof}

Thus while in general there is no algorithm, which, given a simple diagram
$B$, finds the number of maximal paths that $\mu$ almost all  orderings on $B$
have; nevertheless Theorem \ref{generic_three} and
Corollary \ref{generic_one} tell us that one can in principle find this number for a large class of diagrams.

 Next we want to make measure theoretic  statements about perfect subsets  in $(\mathcal
O_{B}, \mu)$: recall that   if $B'$ is a nontrivial telescoping of $B$, the set $L({\mathcal P}_{B})$ is a set
of measure 0 in ${\mathcal P}_{B'}$; for this reason we cannot telescope, and we will use the characterization of perfect orders given by Lemma
\ref{ExistenceVershikMap_Part_2}.
Theorem \ref{generic_two} implies the
following observation for simple diagrams.  If  $B$ is a diagram for which
$\mu (\mathcal O_{B} (j))=1$ with $j>1$,
 then  there is a meagreness of perfect
orderings on $B$ and hence dynamical systems defined on $X_{B}$.
 Part (2) of Theorem \ref{generic_two}
implies an  analogous statement for aperiodic diagrams.

\begin{theorem}
\label{generic_two}
Let $B$ be a finite rank Bratteli diagram.
\begin{enumerate}
\item
Suppose $B$ is simple. If $\mu(\mathcal O_{B}(1)) = 1$, then $\mu({\mathcal P}_{B})=1$.
If $\mu(\mathcal O_{B}(j)) = 1$ for some $j>1$,
then $\mu({\mathcal P}_{B})=0$.
\item
 Suppose that $B$ is aperiodic with $q$ minimal components, and that
 its incidence matrices $(F_{n})$ have a strictly positive row
 $R_{n}$ for each $n$, and where at least one entry in $R_{n}$ tends to $\infty$ as
$n\rightarrow \infty$.  If
 $\mu(\mathcal O_{B}(q)) = 1$, then $\mu(\mathcal P_{B}) = 1$. If
 $\mu(\mathcal O_{B}(j)) = 1$ for some $j>q$, then $\mu(\mathcal
 P_{B}) = 0$.
\end{enumerate}
\end{theorem}

\begin{proof}We remark that if $j=1$, then clearly $\mu$-almost all
  orderings are perfect.

Suppose that $B$ is simple, where there are at most $d$ vertices at each level, and  $\mu(\mathcal O_{B}(j)) = 1$ for some $j>1$.
Fix $0<\delta<1/d$. Define, for $w\, \in V_{n-1}$,
\[P_{n}(w): = \{ v\in V_{n}:m_{v,w}^{(n)}\geq \delta\}\, ;\]
then $V_{n}= \bigcup_{w:P_{n}(w) \neq \emptyset}P_{n}(w)$, and, if for
infinitely many $n$, less than $j$ of the $P_{n}(w)$'s are non-empty,
then, for some $j'<j$, and some $(n_k)$, there is some $\epsilon$ such that $\mu(G_{n_k}^{n_{k+1},j'})\geq
 \epsilon$, and Theorem \ref{generic_three}
implies $\mu({\mathcal O}_{B}(j''))=1$ for $j'' \leq j'<j$, a contradiction.
There is no harm in assuming that for fixed $n$, the sets $\{P_{n}(w):
P_{n}(w)\neq \emptyset\}$ are disjoint - if not we put $v\in P_{n}(w)$, for some $w$
where $m_{v,w}^{(n)} $ is maximal - and that there is some set
$\{w_{1},\ldots w_{j}\}$ of vertices such that $P_{n}(w_i) \neq
\emptyset$ for each natural $n$ and each $i=1,\ldots , j$. If all but finitely many vertices of
the diagram are the range of a bounded number of edges, then Lemma
\ref{generic_one} implies that $\mu(\mathcal O_{B}(1))=1$, a
contradiction. So we can pick
$v_{n}^{*}\in V_{n}$ which has a maximal number of incoming edges. For
ease of notation $v_{n}^{*}=v^{*}$. By the comment just made, we can assume that as $n$ increases, $v^{*}$ is the
range of increasingly many edges.

Let ${\mathcal
E}_{n}$ be the event that
\begin{enumerate}
\item
For each $v\in V_{n}$, the maximal and minimal edge with range $v$ has source
$w_{i}$ whenever $v\in P_{n}(w_{i});$
\item
 for each $n\geq 2$, there is a pair of consecutive edges
with range $v^{*}\in V_{n}$, both having source $w_{i}$ when $v^{*}\in P_{n}(w_{i})$;
\item
For each $n\geq 2$, there is a pair of consecutive edges with range
$v^{*}\in V_{n}$, the first having source $w_{i}$ when $v^{*}\in P_{n}(w_{i})$;
the second  having source $w_{i'}$ for for some $i'\neq i$.
\end{enumerate}

Then there is some $\delta^{*}$ such that $\mu({\mathcal E}_{n})\geq
\delta^{*}$ for all large $n$. So for a set ${\mathcal
O}_{B}(j)'\subset {\mathcal O}_{B}(j)$ of full measure, infinitely
many of the events $\mathcal E_{n}$ occur. For $\om\in {\mathcal O}_{B}(j)'$, if $\om \in \mathcal E_{n}$ for such $n$, then the extremal paths go through the vertices $w_{1},\ldots w_{j}$ at level $n$. Now an application of
Lemma \ref{ExistenceVershikMap_Part_2} implies that  ${\mathcal
O}_{B}(j)'\subset \mathcal O_{B}\backslash \mathcal P_{B}$.

To prove Part 2, first note that if $B$ has $q$ minimal components, then any ordering has at least $q$ extremal pairs of paths. We assume that extremal paths come in pairs - otherwise the ordering is not perfect. If $\mu$-almost all orderings have $q$ maximal paths then necessarily each pair of extremal paths lives in a distinct minimal component of $B$, and $\mu$ almost all orderings  belong to $\mathcal P_{B}$. Suppose that $\mu({\mathcal O}_{B}(j))=1$ where $j>q$.
Write
\[{\mathcal O}_{B}(j) = \bigcup_{\{(k_{1}, \ldots  k_{q}): \sum_{i=1}^{q}k_{i}\leq j\}}
{\mathcal O}_{B}(j, \{(k_{1}, \ldots k_{q})\})\] where ${\mathcal
O}_{B}(j, \{(k_{1}, \ldots ,k_{q})\})$ is the set of orderings with
$k_{i}$ extremal pairs in the $i$-th minimal component. If for some
$i$, $k_{i}>1$, then by the argument in Part (1), $\mu({\mathcal
O}_{B}(j, \{(k_{1}, \ldots ,k_{q})\}))=0$. If $(k_{1},\ldots , k_{q})=
(1,\ldots ,1)$ this means that there is at least one extremal pair of
paths which lives outside the minimal components of $B$. Repeat the
argument in Part 1, except that $v^{*}$ must be chosen outside the
union of the minimal components of $B$, and also such that at least one of the entries in
$\{m_{v^{*},v}^{(n)}: v\in V_{n}\}$
gets large as $n \rightarrow \infty$.

\end{proof}

\begin{example} It is not difficult to find a  simple Bratteli diagram $B$ where almost all orderings are not perfect. Let $V_n=V= \{v_{1},v_{2}\}$ for $n\geq 1$, and  let $\sum_{n=1}^{\infty} m_{v_{i},v_{j}}^{(n)}<\infty$ for $i\neq j$.  Then for $\mu$-almost all orderings, there is some $K$ such that for $k>K$, the sources of the two maximal/minimal edges at level $n$ are distinct - i.e.
$\mu(\mathcal O_{B}(2))=1$. Note that here $\mu(\mathcal O_{B}(2))=1$ if and only if there are two probability measures on $X_{B}$ which are invariant with respect to the tail equivalence relation. This is not in general true as the next example shows.

\end{example}

\begin{example}
This example appears in Section 4 of \cite{ferenczi_fisher_talet:2009}.
Let
\[F_{k}:=\left(
\begin{array}{ccc}
m_{k} & n_{k} & 1 \\
 0 & n_{k}-1 & 1 \\
 m_{k}-1 & n_{k}  & 1\\
\end{array}
\right).
\]
where the sequences $(m_{k})$ and $(n_{k})$ satisfy $3n_{k}+1 \leq 2m_{k} \leq n_{k+1}$,
which implies that they get large. The corresponding stochastic matrix satisfies
\[M_{k}\approx\left(
\begin{array}{ccc}
\frac{m_k}{m_k + n_k} & \frac{n_k}{m_k +n_k} & 0\\
 0 & 1 & 0 \\
\frac{m_k}{m_k + n_k} &  \frac{n_k}{m_k + n_k} & 0\\
\end{array}
\right),
\]
and if we further require that $n_{k+1}\leq Cn_{k}$ for some $C\geq 4$, then
$\frac{n_k}{m_k +n_k}\geq \frac{2}{2+C}$, so that by Corollary \ref{generic_one},
$\mu (\mathcal O_{B}(1)) = 1$, while in \cite{ferenczi_fisher_talet:2009}, it is shown that (a telescoping of) $B$ has 2 probability measures which are invariant under the tail equivalence relation.

\end{example}

\begin{example}\label{firstexample}
Let \[F_{n}:=\left(
\begin{array}{ccccccc}
1 & 1 & 0 & 0 & 0 & 0 & 0 \\
1 & 1 & 0 & 0 & 0 & 0 & 0\\
0 & 0  & 1 & 1 & 0 & 0 & 0   \\
0 & 0  & 1 & 1 & 0 & 0 & 0\\
0 & 0  & 0 & 0 & 1 & 1 & 0  \\
0 & 0  & 0 & 0 & 1 & 1 & 0 \\
1 & 1  & 1 & 1 & 1 & 1 & 1
 \end{array}
\right)
\] for $n$ non-prime and
 \[F_{n}:=\left(
\begin{array}{ccccccc}
1 & 1 & 1 & 1 & 0 & 0 & 0 \\
1 & 1 & 1 & 1 & 0 & 0 & 0\\
1 & 1  & 1 & 1 & 0 & 0 & 0   \\
1 & 1  & 1 & 1 & 0 & 0 & 0\\
0 & 0  & 0 & 0 & 1 & 1 & 0  \\
0 & 0  & 0 & 0 & 1 & 1 & 0 \\
1 & 1  & 1 & 1 & 1 & 1 & 1
 \end{array}
\right)
\] if $n$ is prime.
Then if $n$ is prime, given any vertex $w$, $m_{v,w}^{(n)}\geq 1/7$ either
 for $v=v_{1}$ or $v=v_{5}$. So $\mu(G_{n}^{n+1, 2})\geq
 (1/7)^{7}$. Also $\mu(G_{n}^{n+1,1})=0$ for each $n\geq 1$. Theorem
 \ref{generic_three} implies that $j=2$.
\end{example}

\proof[Acknowledgements]
Most of this work was done during our mutual visits
 to  McGill University, Torun  University, The University of Ottawa and The Institute for Low
 Temperature Physics. We are thankful to these institutions for their
 hospitality and support.  We would like also
 to thank M. Cortez, F. Durand, N. Madras and K. Medynets for useful and stimulating discussions.

\end{document}